\documentclass{amsart} 
\usepackage{amscd,amssymb,amsmath,amsbsy,amsthm}
\usepackage{comment,color}
\usepackage[all]{xy}
\usepackage[colorlinks,plainpages,backref,urlcolor=blue]{hyperref}
\usepackage[width=5.65in,height=7.9in,centering]{geometry}
\usepackage{fancyvrb,bbm}
\usepackage[mathscr]{euscript}
\usepackage{tikz}
\usetikzlibrary{calc}

\newtheorem{theorem}{Theorem}[section]
\newtheorem{lemma}[theorem]{Lemma}
\newtheorem{corollary}[theorem]{Corollary}
\newtheorem{proposition}[theorem]{Proposition}

\theoremstyle{definition}
\newtheorem{definition}[theorem]{Definition}
\newtheorem{exm}[theorem]{Example}
\newtheorem{rem}[theorem]{Remark}
\newtheorem{question}[theorem]{Question}

\newenvironment{example}%
{\pushQED{\qed}\begin{exm}}%
{\popQED\end{exm}}  

\newenvironment{remark}%
{\pushQED{\qed}\begin{rem}}%
{\popQED\end{rem}}  

\newcommand{\abs}[1]{\left|#1\right|}
\newcommand{\set}[1]{\left\{#1\right\}}

\newcommand{\Z}{\mathbb{Z}}
\newcommand{\PP}{\mathbb{P}}
\newcommand{\Q}{\mathbb{Q}}
\newcommand{\R}{\mathbb{R}}
\newcommand{\C}{\mathbb{C}}

\renewcommand{\k}{\Bbbk}  
\newcommand{\one}{\mathbbm{1}}
\newcommand{\X}{\underline{X}}  
\renewcommand{\AA}{\underline{A}}  

\newcommand{\bL}{\mathbf{L}}  

\newcommand{\sV}{\mathsf{V}}
\newcommand{\sW}{\mathsf{W}}

\newcommand{\m}{\mathfrak{m}}

\DeclareMathOperator{\Tor}{Tor}

\DeclareMathOperator{\Ext}{Ext}
\DeclareMathOperator{\Hom}{Hom}
\DeclareMathOperator{\RHom}{{\mathbf R}Hom}
\DeclareMathOperator{\modR}{mod-}

\DeclareMathOperator{\pdim}{pdim}
\DeclareMathOperator{\pd}{pd}  
\DeclareMathOperator{\ab}{ab}
\DeclareMathOperator{\ch}{char}
\DeclareMathOperator{\Sym}{Sym}
\DeclareMathOperator{\gr}{gr}
\DeclareMathOperator{\linn}{\,lin}  
\DeclareMathOperator{\lk}{lk}   
\DeclareMathOperator{\rank}{rank}
\DeclareMathOperator{\corank}{corank}
\DeclareMathOperator{\PD}{PD}   
\DeclareMathOperator{\FP}{FP}   
\DeclareMathOperator{\FL}{FL}   
\DeclareMathOperator{\mspec}{m-Spec}   
\DeclareMathOperator{\supp}{supp}      

\newcommand{\A}{{\mathcal{A}}}

\newcommand{\D}{{\mathbf{D}}}  

\newcommand{\V}{{\mathcal V}}
\newcommand{\W}{{\mathcal W}}
\newcommand{\RR}{{\mathcal R}}
\newcommand{\Hy}{{\mathbb H}} 
\newcommand{\ZZ}{{\mathcal Z}}

\def\dot{\mathchar"013A}  
\newcommand{\hdot}{{\raise1pt\hbox to0.35em{\huge $\dot$}}} 

\title[Abelian duality and propagation of resonance]%
{Abelian duality and propagation of resonance}

\author[G. Denham]{Graham Denham$^1$} 
\address{Department of Mathematics, University of Western Ontario,
London, ON  N6A 5B7, Canada}  
\email{\href{mailto:gdenham@uwo.ca}{gdenham@uwo.ca}}
\urladdr{\href{http://www.math.uwo.ca/~gdenham}%
{http://www.math.uwo.ca/\~{}gdenham}}
\thanks{$^1$Partially supported by an NSERC Discovery Grant (Canada)}

\author[A.~I. Suciu]{Alexander~I.~Suciu$^2$}
\address{Department of Mathematics,
Northeastern University,
Boston, MA 02115, USA}
\email{\href{mailto:a.suciu@neu.edu}{a.suciu@neu.edu}}
\urladdr{\href{http://www.northeastern.edu/suciu/}%
{http://www.northeastern.edu/suciu/}}
\thanks{$^2$Partially supported by 
NSF grant DMS--1010298, NSA grant H98230-13-1-0225, and 
Simons Foundation collaborative grant 354156}

\author[S. Yuzvinsky]{Sergey~Yuzvinsky}
\address{Department of Mathematics, 
University of Oregon, Eugene, OR 97403, USA} 
\email{\href{mailto:yuz@uoregon.edu}{yuz@uoregon.edu}}
\urladdr{\href{http://pages.uoregon.edu/yuz/}%
{http://pages.uoregon.edu/yuz/}}

\begin{document}

\begin{abstract}  
We explore the relationship between a certain ``abelian duality" 
property of spaces and the propagation properties of their cohomology 
jump loci.  To that end, we develop the analogy between abelian 
duality spaces and those spaces which possess what we call the 
``EPY property."   The same underlying homological algebra allows 
us to deduce the propagation of jump loci: in the former case, characteristic 
varieties propagate, and in the latter, the resonance varieties. 
We apply the general theory to arrangements of 
linear and elliptic hyperplanes, as well as toric complexes,  
right-angled Artin groups, and Bestvina--Brady groups.  
Our approach brings to the fore the relevance of the 
Cohen--Macaulay condition in this combinatorial context. 
\end{abstract}


\setcounter{tocdepth}{1}
\maketitle
\tableofcontents

\section{Introduction}
\label{sect:intro}

It has long been recognized that Poincar\'{e} duality imposes 
constraints not just on the integral cohomology of a closed, 
orientable manifold, but also on the cohomology with twisted 
coefficients of such a manifold. Our goal here is to explore 
the constraints imposed by a different type of duality property, 
one which applies to other types of spaces, such as complements 
of hyperplane arrangements and certain polyhedral products. 

\subsection{Duality and abelian duality}
\label{subsec:intro1}
Duality groups were introduced by Bieri and Eckmann \cite{BE}: they
are characterized by a cohomological vanishing condition 
less restrictive than Poincar\'e duality, and they possess a more
general isomorphism between homology and cohomology of group 
representations.  In this paper, we introduce an independent but
related notion which we call abelian duality and which, we see, explains some
previously conjectural behavior of the cohomology of abelian representations.

More precisely, let $X$ be a connected, finite-type CW complex, 
with fundamental group $G$.  Following Bieri and Eckmann \cite{BE}, 
we say that $X$ is a {\em duality space}\/ of dimension $n$ if 
$H^q(X,\Z{G})=0$ for $q\ne n$ and $H^n(X,\Z{G})$ is non-zero 
and torsion-free.   This type of duality property bears
a formal resemblance to Serre duality for Cohen--Macaulay schemes, 
which itself can be understood in terms of Verdier duality.
In section \S\ref{subsec:verdier}, we review Bieri and Eckmann's 
definition in a slightly more general algebraic context in order to 
note such parallels with other homological duality theories.

By analogy, we say that a space $X$ is an {\em abelian duality space}\/ of dimension 
$n$ if the analogous condition, with $\Z{G}$ replaced by $\Z{G}^{\ab}$ 
is satisfied.  In that case, we call the module $D:=H^n(X,\Z{G}^{\ab})$ the
(abelian) dualizing module.  Just as in the classical case, 
the abelian duality property gives rise to duality isomorphisms 
of the form 
\[
H^i(X,A)\cong H_{n-i}(G^{\ab},D\otimes_\Z A),
\] 
for any abelian representation $A$ of $G$ and for all $i$ 
(Proposition~\ref{prop:duality}).  By restricting to abelian representations,
however, the commutativity of the base ring $\Z{G}^{\ab}$ imposes
additional structure and reveals some interesting geometry.
We give examples that show that 
the two properties are independent.  Examples of duality groups which 
are not abelian duality groups are abundant, while examples with the 
latter but not the former property can be contrived.  

Furthermore, we examine the behavior of abelian duality
under fibrations.  We say that a fibration sequence $F\to E\to B$ is
{\em ${\ab}$-exact}\/ if $\pi_1(B)$ acts trivially on $H_1(F,\Z)$, and
the transgression map $H_2(B,\Z)\to H_1(F,\Z)$ is zero.  With this
(essential) hypothesis and a finite-type condition, we show that, if
two out of three of the spaces in the fibration sequence are abelian
duality spaces, then so is the third (Proposition~\ref{prop:extensions}).

\subsection{Propagation of characteristic varieties}
\label{subsec:intro2}
Now let $\k$ be an algebraically closed field, and regard $\k{G}^{\ab}$ 
as the coordinate ring of the one-dimensional representation variety
of $G$, that is, the character group $\widehat{G}:=\Hom(G,\k^*)$.  
Computing homology with coefficients in the rank $1$ local systems 
parametrized by $\widehat{G}$ yields the (depth $1$) {\em characteristic 
varieties}\/ $\V^p(X,\k)$, consisting  of those characters $\rho$ for which 
$H_p(X,\k_{\rho})\ne 0$. The geometry of these varieties controls to a 
large extent the homology of regular, abelian covers of $X$, see 
for instance \cite{DS14, Su14a, Su14b} and references therein. 

Since $X$ is assumed to be connected, $\V^0(X,\k)$ has just 
one point, the trivial character $\one$. The higher-degree characteristic 
varieties, though, may be arbitrarily complicated.  Nevertheless, 
in our situation, a nice pattern holds. 

\begin{theorem}
\label{thm:cvprop-intro}
Suppose $X$ is an abelian duality space of dimension $n$ over $\k$.
Then the characteristic varieties of $X$ propagate: that is,
for any character $\rho\in\widehat{G}$, if $H^p(X,\k_\rho)\neq 0$, 
then $H^q(X,\k_\rho)\neq0$ for all $p\leq q\leq n$.
Equivalently, 
\[
\set{\one}=\V^0(X,\k)\subseteq \V^1(X,\k)\subseteq\cdots\subseteq\V^{n}(X,\k).
\]
\end{theorem}

This result puts some stringent constraints on the Betti numbers of 
such a space $X$, to wit, $b_p(X)> 0$ for $0\le p\le n$, and $b_1(X)\ge n$.   

\subsection{The EPY property and propagation of resonance}
\label{subsec:intro3} 
A similar qualitative phenomenon also appeared in 
a paper of Eisenbud, Popescu and Yuzvinsky in \cite{EPY03},
who showed that the (depth $1$) resonance varieties 
of complex hyperplane arrangement complements ``propagate."  
This result followed from another cohomological 
vanishing property, this time involving the cohomology ring itself.

If $X$ is a space as above, we may regard its cohomology ring 
$A=H^\hdot(X,\k)$ as a module over the exterior algebra $E$ generated
by $V:=H^1(X,\k)$.  For hyperplane complements, these authors prove 
that $A$ admits a linear, injective resolution as a graded 
$E$-module, after a shift of degrees.  They  
express this condition in the elegant terms of the BGG correspondence,
which shows it is equivalent to the statement that $H^p(A\otimes_\k S,d)=0$
for all $p\neq n$, where $S$ is the polynomial ring on $V^*$, and the
differential $d$ is induced by multiplication by the canonical element of
$V\otimes V^*$.  This complex and its vanishing property, it seems, 
should be regarded as an infinitesimal analogue of the equivariant 
cochain complex on the universal abelian cover of $X$ together with its 
cohomological vanishing property.   In view of these considerations, 
we will say that a space $X$ has the {\em EPY property}\/ if 
the shifted cohomology ring $A^*(n)$ has a linear free $E$-resolution, 
for some integer $n$.

Assume now that $\ch(\k)\ne 2$, or that $\ch(\k)=2$ and $H_1(X,\Z)$ has no 
$2$-torsion. For each $a\in V$, then, left-multiplication by $a$ defines 
a cochain complex, $(A,a \cdot)$. The (depth~$1$) {\em resonance varieties}\/ of 
$X$ over the field $\k$ are the jump loci for the cohomology of these complexes; 
that is,  $\RR^{p}(X,\k)=\{a \in V \mid \dim_{\k} H^{p}(A,a\cdot ) \ne 0\}$.  
The next result can be viewed as the infinitesimal analogue of 
Theorem \ref{thm:cvprop-intro}.

\begin{theorem}
\label{thm:resprop-intro}
Suppose $X$ is a finite, connected CW-complex of dimension $n$ 
with the EPY property over a field $\k$.  Then
the resonance varieties of $X$ propagate: that is, for any $a\in A^1$,
if $H^p(A,a\cdot)\neq0$, then $H^q(A,a\cdot)\neq0$ for all $p\leq q\leq n$,
where $A=H^\hdot(X,\k)$.  Equivalently,
\[
\set{0}=\RR^0(X,\k)\subseteq \RR^1(X,\k)\subseteq\cdots\subseteq\RR^{n}(X,\k).
\]
\end{theorem}

Furthermore, if a space $X$ as above admits a minimal CW-structure, 
then, for each $p\le n$, the $I$-adic completion of the module $H^p(X,\k{G^{\ab}})$ 
vanishes (Theorem \ref{thm:kosab}).

\subsection{Arrangements of submanifolds}
\label{subsec:intro4}
Using our previous paper \cite{DSY16}, we show 
that complements of linear and elliptic hyperplane arrangements 
are abelian duality spaces.  It follows from Theorem~\ref{thm:cvprop-intro}
that their characteristic varieties propagate.  
At least for linear arrangements, this phenomenon had been 
observed in examples and had been a matter of speculation.  
Linear hyperplane arrangements are also the first and motivating 
example of spaces with the EPY property.  

\begin{theorem}
\label{thm:arrprop-intro}
Let $\A$ be an arrangement in either $\mathbb{CP}^n$ or $E^{\times n}$, 
and let $U=U(\A)$ be its complement.  In the first case, both the 
characteristic and the resonance varieties of $U$ propagate.
In the latter case, the characteristic varieties of $U$ propagate. 
\end{theorem}

We make note of a special case.  If $\A$ is an arrangement of
(linear) hyperplanes
defined by the vanishing of a homogeneous polynomial $f$ of degree $m$, the 
nonzero level sets of $f$ are all diffeomorphic, and the {\em Milnor fiber}\/ 
of $\A$, denoted $F(\A)$, is the smooth affine hypersurface cut out by
the equation $f=1$.  The Milnor fiber admits an action of the cyclic
group $\Z/m\Z$: if $h$ denotes a preferred generator, it is interesting
to consider the action of $h$ in homology.  If $\zeta$ is an eigenvalue
of $h$ on $H_p(F(\A),\C)$, let $H_p(F(\A),\C)_\zeta$ denote the corresponding
eigenspace.  We show that Theorem~\ref{thm:arrprop-intro} implies that 
\begin{corollary}
\label{cor:MFprop-intro}
If $\A$ is an arrangement of rank $n+1$ and $H_p(F(\A),\C)_\zeta\neq0$
for some $p<n$, then $H_q(F(\A),\C)_\zeta)\neq0$ as well, for all
eigenvalues $\zeta$ and all $p\leq q\leq n$.
\end{corollary}

Using the analysis from the previous section, we show that 
the Milnor fiber itself may be an abelian duality space.  A sufficient
condition is given by requiring that the action of $h$ on $H_1(F(\A),\Z)$
is trivial.

\subsection{Polyhedral products}
\label{subsec:intro5}
Another class of spaces where our approach works rather well is 
that of toric complexes, which are a special instance of generalized 
moment angle complexes, or polyhedral products, see e.g.~\cite{DS07, BBCG07}. 
Given a simplicial complex $L$ with $m$ vertices, the toric complex $T_L$ 
is a subcomplex of the $m$-torus, whose $k$-dimensional subtori correspond to 
the $k$-vertex simplices of $L$. In particular, $\pi_1(T_L)$ is the right-angled Artin group 
associated to the $1$-skeleton of $L$. 

We characterize here which right-angled Artin groups are abelian duality groups,
thus extending the work of Brady, Jensen and Meier~\cite{BM01,JM05}.
It turns out that for right-angled Artin groups, the properties of 
being a duality group, being an abelian duality group, or having the 
EPY property are all equivalent to the Cohen--Macaulay property for
the presentation's flag complex.

\begin{theorem}
\label{thm:toric-intro}
Let $L$ be an $n$-dimensional simplicial complex.  
Then the following are equivalent. 
\begin{enumerate}
\item The toric complex $T_L$ is an abelian duality 
space over $\k$ (of dimension $n+1$).  
\item The cohomology algebra 
$H^\hdot(T_L,\k)$ has the EPY property.
\item The complex $L$ is Cohen--Macaulay over $\k$. 
\end{enumerate}
\end{theorem}

The Bestvina--Brady groups are subgroups of right-angled Artin groups
that parallel the relation between the Milnor fiber and the complement
of an arrangement.
Using the behavior of the abelian duality property under
fibrations again, we find that the Bestvina--Brady groups are also abelian
duality groups if and only if they are duality groups, in parallel with
the analysis of Davis and Okun~\cite{DO12}.   

\subsection{Outline}
\label{subsec:intro_outline}
For the readers' convenience, we give a brief overview of the structure of the
paper.  Sections~\S\ref{sect:cjl} and \S\ref{sect:bgg-koszul} are 
purely algebraic.  We introduce
cohomology jump loci and support loci in \S\ref{sect:cjl}
for chain complexes and prove some preparatory lemmas.  In 
\S\ref{sect:bgg-koszul}, we recall the Bernstein--Gelfand--Gelfand correspondence
in the context of jump loci, and we introduce resonance varieties.

The impatient reader is invited to proceed directly to \S\ref{sect:abdual}, 
where we introduce the abelian duality property and compare it with 
the classical notion of a duality group, and then to \S\ref{sect:cv prop}, 
where we describe the various relationships between abelian duality, the EPY property, 
and propagation.  We conclude by examining in depth two families of motivating examples: 
arrangements of submanifolds in \S\ref{sect:hyparr}, and polyhedral products 
in \S\ref{sect:raag}.

\section{Jump loci and support loci}
\label{sect:cjl}

\subsection{Algebraic preliminaries I}  
\label{subsec:basics}
We begin by singling out some general commutative algebra which
we will use repeatedly.  The reader may refer to \cite{EisenBook}
for details.  We will let $\k$ denote a coefficient ring: unless
specified otherwise, either $\k=\Z$ or a field.  Let $R$ be a 
$\k$-algebra (usually assumed to be commutative and Noetherian), and
let $V=\mspec R$.  We will use the notational convention that, if
$C_\hdot$ is a chain complex, $C^i:=C_{-i}$ for all $i\in\Z$.

Although the next two lemmas
are well-known to experts, for lack of a reference, we
include their proofs.

\begin{lemma}
\label{lem:convex}
If $R$ is a commutative, local $\k$-algebra, then for any finitely-generated
$R$-module $N$, we have $\Tor^R_i(N,\k)\neq0$ for 
$0\leq i\leq\pd_R(N)$, and zero otherwise.
\end{lemma}

\begin{proof}
Set $d=\pd_R(N)$.  
Finitely-generated projective modules over a local ring are free,
by Nakayama's lemma, so $N$ has a free resolution of length $d$ (where
we allow $d=\infty$):
\begin{equation}
\label{eq:naka}
\xymatrix@C-1em{
0&N\ar[l] & F_0\ar[l] & F_1\ar[l] &\cdots\ar[l] & F_d\ar[l] & 0\ar[l]
}.
\end{equation}
We may take this resolution to be minimal \cite[Corollary~19.5]{EisenBook},
in which case $\Tor^R_i(N,\k)\cong F_i\otimes_R\k$.  In particular,
we see that $\Tor^R_i(N,\k)$ is nonzero precisely for $0\leq i\leq d$.
\end{proof}

\begin{lemma}
\label{lem:local}
Suppose $M$ is an $R$-module and $\m\in V$.  
Then 
\[
\Tor^{R_\m}_i(M_\m,\kappa(\m))\neq0 \Leftrightarrow 
\Tor^R_i(M,\kappa(\m))\neq 0,
\]
for all $i\ge0$.
\end{lemma}

\begin{proof}
Localization is exact, so $\Tor^{R_\m}_i(M_\m,\kappa(\m))=
\Tor^R_i(M,\kappa(\m))_\m$, for all $i\ge 0$.
The implication ``$\Rightarrow$'' follows.  

For the converse, we must check
that if $x\in\Tor^R_i(M,\kappa(\m))$ is nonzero, 
then its image in the localization
at $\m$ is also nonzero.  It is sufficient to check that $rx=0$ for $r\in R$ 
implies that either $r\in\m$, or $x=0$.

For this, let $F_{\bullet}\rightarrow M\rightarrow 0$ be a projective
resolution of $M$.  Suppose $c\in F_i\otimes_R \kappa(\m)$ is a 
representative for $x$.  Then $rx=0$ implies $rc=\partial(d)$ for some 
$d\in F_{i+1}\otimes_R\kappa(\m)$.
If $r\not\in\m$, then $rs=1+m$ for some $s\in R$ and $m\in\m$, since $\m$ is 
maximal.  Then $\partial(sd)=(rs)c=c+mc=c$, and so $x=[c]=0$, thereby 
completing the proof.
\end{proof}

\subsection{Algebraic preliminaries II}  
\label{subsec:basics2}
In this section, we briefly note some homological algebra special to 
Hopf algebras.
We let $R$ be a Hopf algebra with antipode $\iota\colon R\to R$, which
we will always assume to be an involution.  Our 
principal example will be the group algebra $R=\k[G]$, where
the antipode is defined on the group elements by $\iota(g)=g^{-1}$.

If $A$ is a right $R$-module, then we may also regard $A$ as a left $R$-module
by letting $r\cdot a:=a\cdot\iota(r)$ for $r\in R$ and $a\in A$,
and analogously if $A$ is a left $R$-module.  If we wish to emphasize the
point, we will write $\iota(A)$ in place of $A$.
Even when $R$ is commutative,
then, we preserve the distinction between left and right modules.
If $A$ and $B$ are left $R$-modules, 
then $A\otimes_\k B$ and $\Hom_\k(A,B)$ are left $R$-modules via
the coproduct and, respectively, the action $(rf)(a) = f(\iota(r)\cdot a)$,
for $f\in \Hom_\k(A,B)$, all $r\in R$, and $a\in A$.

\begin{lemma}
\label{lem:rebalancing}
Let $R$ be a Hopf algebra over $\k$.  Suppose $D$ is a right $R$-module 
and $A$ is a left $R$-module.  
\begin{enumerate}
\item If $D$ is flat over $\k$, then
$\Tor^R_\hdot(D,A)=
\Tor^R_\hdot(\k,D\otimes_\k A)$.
\item If $D$ is projective over $\k$, then
$\Ext_R^\hdot(D,A)=\Ext_R^\hdot(\k,\Hom_\k(D,A))$.
\end{enumerate}
\end{lemma}
\begin{proof}
For the first statement, we claim it is enough to check that 
\begin{equation}
\label{eq:hopf identity}
D\otimes_R A\cong \k\otimes_R(D\otimes_\k A)
\end{equation}
for any $A$.  Indeed, suppose $Q_\hdot$ is a projective resolution of $A$.  
It is easily checked that if $B$ is a flat $R$-module, then so is 
$D\otimes_\k B$; then $D\otimes_\k Q_\hdot$ is a flat resolution
of $A$, and the result follows.

To verify \eqref{eq:hopf identity}, let $\eta$ and $\varepsilon$ denote
the unit and counit, respectively.  We check that the $\k$-bilinear map 
induced by $(c,d\otimes a)\mapsto \eta(c)d\otimes a$ is $R$-bilinear.
For this, 
we note $(1\cdot r,d\otimes a)\mapsto \eta\varepsilon(r)d\otimes a$.  
On the right, we compute with Sweedler notation, as in \cite{Sw69}:
\begin{align*}
(1,r\cdot(d\otimes a)) & \mapsto \sum d\iota(r_{(1)})\otimes r_{(2)}a\\
&=\sum d\iota(r_{(1)})r_{(2)}\otimes a\\
&=\eta\varepsilon(r) d\otimes a.
\end{align*}

For the second statement, we check that
\begin{equation}
\label{eq:hopf bis}
\Hom_R(D,A)\cong \Hom_R(\k,\Hom_{\k}(D,A)),
\end{equation}
using a similar argument, which we omit.
\end{proof}

\subsection{Jump loci of chain complexes}
\label{subsec:jump_loci}
Various notions of jump loci for rank-$1$ local systems, resonance varieties
of graded commutative algebras, and support varieties of Alexander invariants 
of spaces are all instances of the following basic notions.  Here, we assume 
that $\k$ is an algebraically closed field, and $R$ is commutative.

\begin{definition}
\label{def:charvar}
If $C_\hdot$ is a chain complex of finitely-generated
$R$-modules, we define its {\em homology jump loci}\/ as follows:
\begin{equation}
\label{eq:cvar}
\V_{i,d}(C_\hdot):=\set{\m\in V\colon \dim_{\kappa(\m)}
H_i(C\otimes_R \kappa(\m))\geq d},
\end{equation}
for all integers $d\geq0$ and all $i$.  Here, $\kappa(\m):=R/\m$.%
\end{definition}

\begin{definition}
\label{def:suppvar}
The {\em homology support loci}\/ of $C_\hdot$ are, by definition,
\begin{equation}
\label{eq:supploci}
\W_{i,d}(C_\hdot) := \supp \Big( \bigwedge^d H_i(C_\hdot) \Big)\subseteq V.
\end{equation}
\end{definition}

Clearly, the sets $\W_{i,d}(C_\hdot)$ depend only on the quasi-isomorphism 
class of the chain complex $C_{\hdot}$.

In both cases, if we wish to regard $C_\hdot$ as a cochain complex,
we will write $\W^{i,d}(C^\hdot):=\supp \big(\bigwedge^d H^i(C^\hdot)\big)$, 
and use the notation $\V^{i,d}(C^\hdot)$ analogously.  
We abbreviate $\V_i(C_\hdot):=\V_{i,1}(C_\hdot)$, and likewise 
for $\W_i(C_\hdot)$, etc.

\begin{remark}
\label{rem:z-closed}
By construction, the sets $\W_{i,d}(C_\hdot)$ are Zariski-closed subsets 
of $V=\mspec R$.  On the other hand, there are chain complexes $C_\hdot$ 
for which the sets $\V_i(C_\hdot)$ are not Zariski-closed, 
see e.g.~\cite[Example 2.4]{PS14}.  Nevertheless, if $C_\hdot$ is 
a chain complex of {\em free}\/ (finitely-generated) $R$-modules, 
then all the jump loci $\V_{i,d}(C_\hdot)$ are Zariski-closed.
\end{remark}

In view of the above remark, we need to modify the above definition of 
jump loci in the case when one of the $R$-modules $C_i$ is not free.  
By a result of Mumford (see \cite[III.12.3]{Hart}), there is a chain complex 
$F_{\hdot}$ of free $R$-modules which is quasi-isomorphic to $C_{\hdot}$; 
moreover, if $C_{\hdot}$ is bounded below, then $F_{\hdot}$ may be 
chosen to be bounded below, too. Finally, 
as shown by Budur and Wang \cite{BW}, the varieties 
$\V_{i,d}(C_\hdot):=\V_{i,d}(F_\hdot)$ 
depend only on $C_{\hdot}$, and not on the choice of $F_{\hdot}$.  

The various notions defined above are different, in general.  Nevertheless, 
as noted in \cite{PS14}, there is a spectral sequence which allows us  
to relate the jump loci $\V_i(C_\hdot)$ and the support loci $\W_i(C_\hdot)$. 
For completeness, we reprove this result, in a slightly more 
general context.

\begin{proposition}[\cite{PS14}, Cor.~4.3]
\label{prop:compare}
Suppose that $C_\hdot$ is bounded below.  Then, for each $j\in\Z$,%
\begin{equation}
\label{eq:equal_unions}
\bigcup_{i\leq j}\V_i(C_\hdot)=\bigcup_{i\leq j}\W_i(C_\hdot).
\end{equation}
\end{proposition}
\begin{proof}
By the above discussion, we may assume $C_\hdot$ is 
a chain complex of free modules.  Since $C_\hdot$ is bounded below, 
there is a spectral sequence
starting at $E^2_{pq}=\Tor^R_p(H_q C,\kappa(\m))$ and 
converging to $H_{p+q}(C_\hdot\otimes_{R} \kappa(\m))$ 
for all $\m\in V$.  So if $\m\in \V_j(C_\hdot)$ for some $j$,
it must be the case that
$E^2_{p,j-p}\neq 0$ for some $p\geq 0$.  It follows from
Lemma~\ref{lem:local} that $(H_{j-p} C)_\m\neq 0$: i.e., $\m\in 
\W_{j-p}(C_\hdot)$, which gives the forward containment in 
\eqref{eq:equal_unions}.

On the other hand, suppose $j$ is the least index for which 
$\m\in \W_j(C_\hdot)$.  By Lemma~\ref{lem:local} again, 
we have that $(E^2_{pq})_\m=0$ for all $q<j$ (and $p<0$).  By hypothesis,
$E^2_{0j}\neq0$, so $E^\infty_{0j}\neq0$ as well.  That is,
$\W_j(C_\hdot)\setminus \bigcup_{i<j}\W_i(C_\hdot)$ 
is included in $\V_j(C_\hdot)$, from which the other containment follows.
\end{proof}

\subsection{Jump loci and cohomological vanishing}
\label{subsec:vanishing}
We now consider a special case, the jump loci of
projective resolutions.
If $C_\hdot$ is a chain complex of $R$-modules, 
for $n\in\Z$ let $C[n]_i:=C_{n+i}$ the shifted complex. 
In this case, the jump loci are nested.

\begin{proposition}
\label{prop:general_propagation}
Suppose that $R$ is commutative and $C_\hdot$ is a complex of finitely-generated
projective modules for which $C_i=0$ for $i<n$, for some $n$,
and $H_i(C)=0$ except for $i=n$.  Then 
\begin{equation}\label{eq:general_propagation}
V\supseteq \V_n(C_\hdot)\supseteq \V_{n+1}(C_\hdot) \supseteq 
\V_{n+2}(C_\hdot) \supseteq \cdots .
\end{equation}
\end{proposition}
We will refer to the phenomenon \eqref{eq:general_propagation}, when it
occurs, as {\em propagation}\/ of homology jump loci.

\begin{proof}
Clearly it is sufficient to consider the case $n=0$.  Let $D=H_0(C)$.
Suppose that $\m\in\V_i(C_\hdot)$ for some $i>0$ and $j$ is an integer
for which $n\leq j\leq i$.  
Since $C_\hdot$ is a resolution, 
\begin{align*}
\m\in \V_i(C_\hdot) &\Leftrightarrow \Tor^R_{i}(D,\kappa(\m))\neq 0\\
&\Leftrightarrow \Tor^{R_{\m}}_{i}(D_{\m},\kappa(\m))\neq 0\\
&\Rightarrow\Tor^{R}_{j}(D,\kappa(\m))\neq 0\\
&\Rightarrow \m\in \V_j(C_\hdot),
\end{align*}
by Lemma~\ref{lem:local} and then Lemma~\ref{lem:convex}.  The claim follows.
\end{proof}
In our applications, the module $D$ in the proof above will be the dualizing
module.
\subsection{Propagation and duality}
\label{subsec:propdual}
If $C$ is a projective resolution, we can also describe jump loci 
in terms of 
homology support loci.  Again, suppose $R$ is a $\k$-Hopf algebra with
antipode $\iota$.  Let $-^*:=\Hom_R(-,\k)$ and $-^{\vee}:=\Hom_R(-,R)$.

\begin{proposition}
\label{prop:UCT}
Suppose $C_\hdot$ is a chain complex of finitely-generated, projective right
$R$-modules
over a field $\k$.
Then 
\begin{equation}
\label{eq:universal_coeffs}
\iota(\V_{i,d}(C_\hdot))=\V^{i,d}(C^\vee),
\end{equation}
for all integers $i,d$.
\end{proposition}
\begin{proof}
Since $C_\hdot$ is finitely-generated and projective, 
for any right $R$-module $A$, the natural map 
\begin{equation}\label{eq:fgdual}
\xymatrixcolsep{18pt}
\xymatrix{\Hom_R(C_\hdot,R)\otimes_R \iota(A)\ar[r]& \Hom_R(C_\hdot,A)}
\end{equation}
is an isomorphism.  Using the Hom-tensor adjunction, for any $\m\in V=\mspec R$,
\begin{align}
\label{eq:crk}
(C_\hdot\otimes_R \kappa(\m))^* &\cong \Hom_R(C_\hdot,\kappa(\m)^*)\\ \notag
&\cong C^\vee\otimes_R\kappa(\iota(\m))
\end{align}
via the isomorphism above.  Since $\k$ is a field, 
the Universal Coefficients Theorem implies this
is an isomorphism in homology, which completes the proof.
\end{proof}

\begin{proposition}
\label{prop:VfromW}
Suppose that $C_\hdot[n]$ is a finitely-generated projective 
resolution for some $n\in \Z$.  Then
\[
\V_k(C_\hdot)=\V_{k+1}(C_\hdot)\cup\iota(\W^k(C^\vee))
\]
for all $k\geq n$.
\end{proposition}
\begin{proof}
We apply Proposition~\ref{prop:compare} to $C^\vee$,
since (by hypothesis) $C^\vee$ is also bounded below.  For all integers $k$,
\begin{align*}
\bigcup_{i\geq k}\iota(\W^i(C^\vee)) &= 
\bigcup_{i\geq k}\iota(\V^i(C^\vee))\\
&= \bigcup_{i\geq k}\V_i(C_\hdot),
&&\text{by Proposition~\ref{prop:UCT},}\\
&= \V_k(C_\hdot)
&&\text{provided $k\geq n$,}
\end{align*}
using Proposition~\ref{prop:general_propagation}.  The result follows.
\end{proof}

\section{The BGG correspondence, Koszul modules, and resonance}
\label{sect:bgg-koszul}

\subsection{The BGG correspondence}
\label{subsec:bgg}
Let $V$ be a finite-dimensional vector space over a field $\k$.
Let $V^*$ denote the dual vector space. The
Bernstein--Gelfand--Gelfand correspondence is an explicit 
equivalence of bounded 
derived categories of graded modules over the exterior
algebra $E:=\bigwedge V$ and over the symmetric algebra 
$S:=\Sym V^*$.

Following the exposition in \cite{Eis05}, we assign the nonzero
elements of $V^*$ degree $1$; then $S_1=V^*$ and $E_{-1}=E^{1}=V$.  
If $M$ is a graded module and $r\in\Z$, we denote a degree shift by letting
$M(r)_i:= M_{r+i}$: then $M(r)^{i}=M^{i-r}$.

Let $\bL$ denote the functor from the category of
graded $E$-modules to the category of linear free complexes
over $S$, defined as follows: for a graded $E$-module $P$, 
$\bL(P)$ is the complex
\begin{equation}
\label{eq:lp}
\xymatrix{
\cdots \ar[r]& P_i \otimes_{\k} S \ar[r]^(.45){d_i} &  P_{i-1} \otimes_{\k} S
 \ar[r]& \cdots
},
\end{equation}
with an $S$-linear differential $d$ induced by left-multiplication by the
canonical element $\omega\in V\otimes V^*$. 
We note a useful fact about the functor $\bL$.
\begin{lemma}[\cite{Eis05}, Ex 7F7]
\label{lem:L_dual}
If $P$ is a finitely-generated $E$-module, then
$\bL(P)^\vee=\bL(P^*)$.
\end{lemma}

\subsection{Resonance varieties}
\label{subsec:res}
Here is another important example of (co)homology jump loci.
Let $P$ be a finitely-generated $E$-module.  For each 
$a\in E^1=V$, we note that $a^2=0$, so left-multiplication by 
$a$ defines a cochain complex 
\begin{equation}
\label{eq:aomoto}
\xymatrix{(P ,  a\cdot)\colon  \ 
\cdots\ar[r] & P^{i-1}\ar^(.55){a}[r] & P^i\ar^{a}[r] & P^{i+1}  \ar[r]& \cdots}.
\end{equation}

We note that the complex $(P, a\cdot)$ is a specialization of the construction
\eqref{eq:lp}: that is, if we make the identification $V=\mspec(S)$, it
follows easily that 
$(P, a\cdot)=\bL(P)\otimes_S\kappa(a)$.  This makes $\bL(P)$ a 
``universal complex'' that parameterizes the complexes $(P,a\cdot)$.  The
corresponding jump loci in this case are known as `resonance' varieties.

\begin{definition}
\label{eq:resvar}
If $P$ is a finitely-generated graded $E$-module, its {\em resonance varieties}\/ 
are defined to be
\begin{align}
\label{eq:rvs}
\RR^{i,d}(P)&:=\V^{i,d}(\bL(P))\\
&=\set{a \in V \colon \dim_{\k} H^i(P, a\cdot ) \ge  d}.\nonumber
\end{align}
As usual, let $\RR^i(P):=\RR^{i,1}(P)$.
\end{definition}

Clearly, these sets are homogeneous algebraic subvarieties 
of the affine space $V$.  Moreover, if $P^i=0$, then 
$\RR^{i,d}(P)=\emptyset$, for all $d>0$.  Our convention is to
index these jump loci cohomologically, since this is the natural
choice for examples in the literature.

\subsection{Propagation of resonance}
\label{subsec:pr}
In this section, we recall some equivalent conditions under which the 
BGG correspondence becomes particularly simple.  
\begin{theorem}[\cite{Eis05}, Thm.~7.7]
\label{thm:linres}
For a finitely-generated graded $E$-module $P$ with $P_i=0$ for $i<0$, 
the following are equivalent.
\begin{enumerate}
\item $P^*$ has a linear free resolution.
\label{it:linres one}
\item $H_i(\bL(P))=0$ for $i\neq 0$.
\label{it:linres two}
\end{enumerate}
\end{theorem}
\begin{definition}
If $P$ satisfies the hypotheses above, the $E$-module $P^*$ is called a 
{\em Koszul module}.  In this case, we let $F(P):=H_0(\bL(P))$, an $S$-module.
\end{definition}

With this, we obtain a special case of 
Proposition~\ref{prop:general_propagation}.

\begin{proposition}
\label{prop:BGG_propagation}
Suppose that $P$ is an $E$-module satisfying one of the equivalent conditions of 
Theorem~\ref{thm:linres}.  Then 
\begin{equation}\label{eq:BGG_propagation}
V\supseteq \RR^0(P)\supseteq \RR^{-1}(P)\supseteq \RR^{-2}(P)\supseteq \cdots.
\end{equation}
\end{proposition}

\begin{proof}
We simply invoke Proposition~\ref{prop:general_propagation}, with
$C_\hdot=\bL(P)$, and $n=0$.
\end{proof}
Specializing further, consider the case of a cyclic, graded $E$-module, 
$A:=E/I$.  Following the development in \cite{EPY03}, we make the following
definition.

\begin{definition}
\label{def:eps}
A cyclic, graded $E$-module $A$ has the {\em EPY property}\/ 
if $A^*(n)$ is a Koszul module for some integer $n$. 
 If $X$ is a finite-type CW-complex of dimension $n$
and $A=H^\hdot(X,\k)$, we say $X$ has the EPY property.
\end{definition}

If $X$ has the EPY property, we note that the integer 
$n$ must be the socle degree of $A$, regarded as a $\k$-algebra.
Another immediate consequence is the following.
\begin{proposition}\label{prop:EPYbounds}
If $X$ has the EPY property over $\k$, then $\dim_\k H^i(X,\k) \geq
{n\choose i}$. 
\end{proposition}
\begin{proof}
By Theorem~\ref{thm:linres}, $\bL(A(-n))$ is a linear, free resolution of
the $S$-module $F(A(-n))$, so $\dim_\k A^{n-i}$ is its $i$th Betti number.
This lower bound on Betti numbers in a linear resolution is due to 
Herzog and K\"uhl~\cite{HK84}.
\end{proof}

\begin{theorem}
\label{thm:resprop}
Suppose $A$ has the EPY property, with socle degree $n$.  
Then the resonance varieties of $A$ propagate:
\[
\RR^i(A)\subseteq\RR^{i+1}(A)
\]
for $0\leq i<n$.
\end{theorem}
\begin{proof}
We apply Proposition~\ref{prop:BGG_propagation} 
to the module $P=A(-n)$, and note
that $\RR^{i}(A(-n))=\RR^{i+n}(A)$, for $i\in\Z$.
\end{proof}

\begin{example}
\label{ex:EisEPY}
The free $E$-module of rank $1$ has the EPY property, since $E^*(n)\cong E$;
thus the $n$-torus $T=(S^1)^{\times n}$ is a space with the EPY property.  It is
easily seen that $\RR^i(E)=\set{0}$ for $0\leq i\leq n$.
\end{example}

Note that $S$ is a Hopf algebra with antipode defined in degree $1$ by
$x\mapsto -x$.  Since the resonance varieties are homogeneous, 
Proposition~\ref{prop:UCT} simplifies to the statement that
\begin{align}
\label{eq:ripd}
\RR^{i,d}(P)&=\V_{i,d}(\bL(P)^\vee)\\ \notag
&=\RR_{i,d}(P^*). 
\end{align}

Let us also consider the relationship between the jump loci and support loci
in the context of propagation.  If $P^*$ is a Koszul module, then
by Theorem~\ref{thm:linres}, $\bL(P)$ resolves $F(P)$.  By 
Lemma~\ref{lem:L_dual}, then,
\begin{align}
\label{eq:hipd}
H^i(\bL(P^*)) &=H^i(\bL(P)^\vee)\\ \notag
&=\Ext^i_S(F(P),S).
\end{align}
Consequently, Proposition~\ref{prop:VfromW} takes the following form.

\begin{proposition}
\label{prop:VfromW_again}
Suppose that $P^*$ is a Koszul $E$-module.  Then
\[
\RR^{-k}(P)=\RR^{-k-1}(P)\cup\supp\Ext_S^{k}(F(P),S)
\]
for all $k\geq0$.  
\end{proposition}

In the special case where $A$ is a cyclic $E$-module with socle degree $n$,
this says
$\RR^{k}(A)=\RR^{k-1}(A)\cup\supp\Ext_S^{n-k}(F(A),S)$, for
$0\leq k\leq n$.  We refer to Example~\ref{ex:CJL_vs_support} to 
see this illustrated explicitly.

\section{Duality and abelian duality spaces}
\label{sect:abdual}

\subsection{Duality groups}
\label{subsec:dgroups}
We now recall a well-known notion, due to Bieri 
and Eckmann \cite{BE}.  Let $G$ be a group, and let 
$\Z{G}$ be its group-ring. For simplicity, we will assume 
$G$ is of type $\FP$, i.e., there is a finite resolution 
$P_{\hdot} \to \Z \to 0$ of the trivial $\Z{G}$-module $\Z$ 
by finitely generated projective $\Z{G}$-modules.

\begin{definition}  
\label{def:duality group}
An $\FP$-group $G$ is called a {\em duality group}\/ of 
dimension $n$ if $H^p(G,\Z{G})=0$ for $p\ne n$ and $H^n(G,\Z{G})$ 
is non-zero and torsion-free.
\end{definition}

If $G$ satisfies the above condition, the abelian group 
$C=H^n(G,\Z{G})$, viewed as a (right) module over $\Z{G}$, 
is called the {\em dualizing module}\/ of $G$.  Then $n$ is 
the cohomological dimension of $G$, and the hypotheses
imply that the complex
\begin{equation}
\label{eq:cc}
\xymatrixcolsep{20pt}
\xymatrix{0\ar[r]& P_0^{\vee}\ar[r] & P_1^{\vee}\ar[r] & \cdots\ar[r] & 
P_n^{\vee}\ar[r] & C\ar[r] & 0
}
\end{equation}
is a finitely generated projective resolution of $C$.  Using the identification
\eqref{eq:fgdual}, for all $i$ we have
\begin{align}
H^i(G,A) & \cong \Tor^{\Z G}_{n-i}(C,A) \nonumber \\
&\cong H_{n-i}(G,C\otimes_\Z A), \label{eq:duality gp iso}
\end{align}
by Lemma~\ref{lem:rebalancing}.

By a celebrated theorem of Stallings and Swan, the duality 
groups of dimension $1$ are precisely the finitely generated 
free groups.  It is also known that torsion-free, one-relator 
groups are duality groups of dimension $2$.  
The class of duality groups is closed under extensions and 
passing to finite-index subgroups.  On the other hand, a free 
product of duality groups is not a duality group, unless all the 
factors are free, cf.~\cite{Br}.  

\begin{remark}
\label{rem:pid}
If we replace $\Z$ by a principal ideal domain $\k$ in 
Definition~\ref{def:duality group}, we will say $G$ is {\em a duality group
over $\k$}.  We use the same terminology in the definitions that follow.
We will consider only the cases where
$\k=\Z$ or a field.  Clearly, if $G$ is a duality group, then it is a 
duality group over any choice of $\k$.
\end{remark}

In \cite[Theorem 6.2]{BE} Bieri and Eckmann provide a geometric 
criterion for a group $G$ to be a duality group.  Namely, 
suppose $G$ admits as classifying space a compact, connected, 
orientable manifold-with-boundary $M$ of dimension $m$, and suppose 
$\partial \widetilde{M}$, the boundary of the universal cover 
of $M$, has the integral homology of a wedge of $q$-spheres.  
Then $G$ is a duality group of dimension $n=m-q-1$, 
with dualizing module $C=H_q(\partial \widetilde{M},\Z)$.

Of note is the situation when the 
dualizing module, $C=H^n(G,\Z{G})$, is the trivial 
$\Z{G}$-module $\Z$. In this case, $G$ is an (orientable) 
{\em Poincar\'e duality group}\/ of dimension $n$---for short, 
a $\PD_n$-group.  For instance, if $G$ admits a 
classifying space which is a closed, orientable, 
manifold $M$ of dimension $n$ 
(this is the case $\partial M=\emptyset$ and $q=-1$ from 
above), then $G$ belongs to this class. 
If $G$ is a duality group that is not a Poincar\'e duality group, 
then the dualizing module $C$ contains no non-zero $\Z{G}$-submodule 
which is finitely generated over $\Z$, cf.~\cite{Fa75}.  

\subsection{Bieri--Eckmann duality in context}
\label{subsec:verdier}
Bieri--Eckmann duality for representations of discrete groups bears
a formal resemblance to Serre duality for coherent
sheaves on projective, Cohen--Macaulay schemes.  Since the latter can
be understood in terms of Verdier duality, as explained in \cite{Ha66},
it is natural to examine the former in this light as well.  A general
framework for such duality theories is provided by Boyarchenko and 
Drinfeld~\cite{BD13}.  

To start, for any $\FP$-group $G$, there is a duality isomorphism in the
derived category that takes on the stronger form described in the previous
section in the special case
when $G$ is a duality group.  That is, 
for any coefficient ring $\k$, let $R=\k{G}$, and 
let $\omega_G:=\RHom(\k,R)$.  Since $\k$ has a 
finite projective resolution, $\omega_G\in \D^b(\modR\!R)$.  
Since that resolution can be taken to be finitely generated, and 
for finitely generated projective modules $P$ there is an isomorphism
$\Hom_R(P,R)\otimes_R -\cong \Hom_R(P,-)$, we have
an isomorphism of functors on $\D^b(\modR\!R)$ generalizing the duality
group isomorphism \eqref{eq:duality gp iso}:
\begin{align}
\RHom_R(\k,-) &\cong \omega_G\otimes^{\bL}_R-\\  \notag
&\cong \k \otimes_R^{\bL}(\omega_G\otimes^{\bL}_{\k}-), 
\end{align}
using Lemma~\ref{lem:rebalancing} again.  In this language,
$G$ is a $\k$-duality group of dimension $n$ 
if and only if $\omega_G$ is quasiisomorphic to a complex concentrated
in cohomological degree $n$.

To move beyond this purely algebraic setting, let us now assume that $G$
is a $\FL$ group; that is, $G$ admits a compact classifying space $BG$,
with universal cover $EG$.  Let $\k$ be a field, and take 
\cite{Iv86, KS90, Dim04} as references for sheaf theory.
As observed in \cite[p.\ 209]{Br}, then $H^\hdot(G,\k[G])=H^\hdot_c(EG,\k)$.
Let $f$ denote the constant function on $EG$.  Then Verdier duality implies
an isomorphism in cohomology
\begin{equation}\label{eq:verdier}
H^p_c(EG,\k)^{\vee} \cong \Hy^{-p}(EG,f^!\k)
\end{equation}
for all $p$, where $f^!\k\in \D^b(EG)$ is the dualizing complex on $EG$. 
That is, $G$ is a $\k$-duality group of dimension $n$
if and only if the hypercohomology \eqref{eq:verdier} vanishes for $p\neq n$.
In the case where $f^!\k$ is a complex of local systems (for example, if 
$EG$ is a manifold), the hypercohomology spectral sequence degenerates, so
$G$ is a $\k$-duality group if and only if $f^!\k$ is represented by a 
complex of sheaves which is nonzero only in cohomological dimension $n$.

\begin{remark}
\label{rem:nonex}
In general, the dualizing complex is constructible but not a local system,
in which case the property that $f^!\k\cong \omega[n]$ for some sheaf $\omega$
on $EG$ does not imply that $G$ is a duality group.  Indeed,
suppose $G$ admits as classifying space an aspherical, orientable 
manifold-with-boundary $M^m$, and let
$j\colon \widetilde{M} \setminus \partial\widetilde{M}\to \widetilde{M}$ 
be the inclusion map.  Then
$f^!\k=j_!\tilde{\k}[m]$, where $\tilde{\k}$ is the constant sheaf: see, e.g.,
\cite[p.~298]{Iv86}.  Yet $G$ need not be a duality group: for instance, 
take $G=F_2*\Z$, and let $BG=M^4$ be a thickening of $T^2\vee S^1$ 
in $\R^4$.
\end{remark}

Finally, as noted by Bieri and Eckmann, there is an intriguing analogy 
in the context of profinite groups between their notion 
of duality groups and what Verdier calls ``Cohen--Macaulay 
groups'' in the last chapter of \cite{Serre}.  This analogy 
was explored more fully by Pletsch in \cite{Pl80}.  It would be interesting
to know if there could be an arithmetic analogue of abelian duality and its
consequences for local systems considered in Section~\ref{sect:cv prop}.

\subsection{Duality spaces}
\label{subsec:dspaces}
The notion introduced in Definition \ref{def:duality group} can be 
extended from groups to spaces. 
One possible definition of a duality space is given in \cite[\S~6.2]{BE}.  
For our purposes, we adopt a slightly different definition, inspired 
by recent work of Davis, Januszkiewicz, Leary, and Okun \cite{DJLO11}.

Let $X$ is a path-connected space with the homotopy type 
of a CW-complex.  We shall assume $X$ is of finite-type, i.e., 
it has finitely many cells in each dimension.  
Without loss of generality, we may assume $X$ has a 
single $0$-cell, say, $x_0\in X$. Let $G=\pi_1(X,x_0)$ 
be the fundamental group of $X$. 

\begin{definition}
\label{def:duality space}
A space $X$ as above is called a {\em duality space}\/ of dimension $n$ if
$H^p(X,\Z{G})=0$ for $p\ne n$ and $H^n(X,\Z{G})$ is non-zero and 
torsion-free.
\end{definition}

As we noted in the aspherical case, this is equivalent to the statement that
the compactly supported cohomology of 
the universal cover, $\widetilde{X}$, is concentrated in a 
single dimension (where it is torsion-free).  Clearly, if $X$ 
is aspherical, then its fundamental group is a duality group.
Generalizing the discussion above slightly shows that, if $X$ is a 
duality space of dimension $n$ and $C=H^n(X,\Z{G})$, then
\begin{equation}\label{eq:dualityspaces}
H^i(X,A)\cong H_{n-i}(G,C\otimes_\Z A)
\end{equation}
for all $i\geq0$ and $\Z{G}$-modules $A$.

\subsection{Abelian duality spaces}
\label{subsec:abdspaces}
Let $X^{\ab}$ be the universal abelian cover of $X$.  The 
fundamental group of $X^{\ab}$ is $G'$, the commutator 
subgroup of $G$, while the group of deck transformations 
is the abelianization, $G^{\ab}=G/G'$. 
We introduce now a variation on the above definition, 
which essentially replaces the universal cover by the universal 
abelian cover. 

\begin{definition}
\label{def:abeldual}
The space $X$ is called an {\em abelian duality space}\/ of 
dimension $n$ if $X$ is homotopy equivalent to a finite, connected 
CW-complex of dimension $n$, 
$H^p(X,\Z{G^{\ab}})=0$ for $p\ne n$, and $H^n(X,\Z{G}^{\ab})$ 
is non-zero and torsion-free.  If in addition $X$ is aspherical, then 
we will call $G:=\pi_1(X,*)$ an {\em abelian duality group}\/ of
dimension $n$.
\end{definition}

In that case, the group $D=H^n(X,\Z{G^{\ab}})$, 
viewed as a module over $\Z{G^{\ab}}$, is called the 
{\em dualizing module}\/ of $X$.  We shall sometimes 
refer to these spaces as $\ab$-duality spaces.  Just as in 
\S\ref{subsec:verdier}, we can reformulate this
notion as a property of the universal abelian cover.  Once again, 
let $f$ denote the constant map on $\widetilde{X}^{\ab}$: then $X$ is an 
abelian duality space if and only if $\Hy^{p}(\widetilde{X}^{\ab},f^!\Z)=0$
for $p\neq -n$, and is nonzero and torsion-free for $p=-n$.

As an example, finitely generated free groups are 
abelian duality groups.  On the other hand, surface groups 
of genus at least $2$ are {\em not}\/ abelian duality groups 
(see Example \ref{ex:surfaces}), though of course they are 
(Poincar\'e) duality groups.   Next, we give an example 
(suggested by Ian Leary) of an abelian duality group  
which is not a duality group:

\begin{example}
\label{ex:leary}
Let $H=\langle x_1,\dots , x_4 \mid x_1^{-2} x_2 x_1 x_2^{-1}, \dots, 
x_4^{-2} x_1 x_4 x_1^{-1}\rangle$ be Higman's acyclic group, 
and let $G=\Z^2 * H$.  Then $G$ is an abelian duality group 
(of dimension $2$), but not a duality group.
\end{example} 

Of note is the situation when the dualizing module, 
$D=H^n(G,\Z{G^{\ab}})$, is the trivial $\Z{G^{\ab}}$-module $\Z$. 
Let us call such a group a {\em Poincar\'{e} abelian duality group}\/ 
of dimension $n$.  Such groups (with $n=2$) were studied by 
Fenn and Sjerve in \cite{FS82}.  We extend the notion as before
to spaces.

\begin{question}  
\label{quest:dualizing mod}
Suppose $G$ is an abelian duality group that is not a Poincar\'{e} abelian 
duality group.  Does the dualizing module $D$ contain any non-zero 
$\Z{G}^{\ab}$-submodule which is finitely generated over $\Z$? 
\end{question}

As the name is intended to suggest, abelian duality spaces have a
(co)homological property which is analogous to the classical 
one of \eqref{eq:duality gp iso}.  For convenience, let $R=\Z G^{\ab}$.

\begin{proposition}
\label{prop:duality}
Suppose $X$ is an abelian duality space of dimension $n$ with dualizing module
$D$,  and $A$ is a (left) $R$-module.  Then 
\begin{equation}
\label{eq:abdualityiso}
H^i(X,A)\cong \Tor^{R}_{n-i}(D,A)\cong H_{n-i}(G^{\ab},D\otimes_{\Z} A)
\end{equation}
for all $i\geq0$.  If, moreover, $D$ is a free abelian group, then
\begin{equation}\label{eq:abdualityiso2}
H_i(X,A) \cong H^{n-i} (G^{\ab}, \Hom_{\Z}(D,A)).
\end{equation}
\end{proposition}

\begin{proof}
By hypothesis, $C^\hdot(\widetilde{X}^{\ab})[n]$ is a finitely-generated
complex of free $R$-modules, 
so we may identify
$\Hom_R(C_\hdot(\widetilde{X}^{\ab}),A)$ with 
$C^\hdot(\widetilde{X}^{\ab})\otimes_R A$.  Since 
$C^\hdot(\widetilde{X}^{\ab})[n]$ is a resolution of $D$, we may now apply
Lemma~\ref{lem:rebalancing}.  The second claim is similar.
\end{proof}

One consequence is that the Poincar\'e abelian duality property is rather
rare. 

\begin{corollary}
\label{cor:pad}
Suppose $X$ is both a Poincar\'e duality space and an abelian duality space.  Then 
$X$ is a Poincar\'e abelian duality space.  Moreover, for such $X$ 
the commutator subgroup of $\pi_1(X)$ is perfect.
\end{corollary}

\begin{proof}
Using the first property, and the isomorphism \eqref{eq:dualityspaces},
we see that $D=H^n(X,\Z G^{\ab})=H_0(X,\Z G^{\ab})=\Z$, the trivial module, 
and this proves the first claim.  

Now we assume $X$ is a Poincar\'e abelian duality space. 
Using now the isomorphism \eqref{eq:abdualityiso2}, we see that 
\begin{align*}
H_p(X, \Z G^{\ab})&=H^{n-p}(G^{\ab},\Z G^{\ab})\\
&=0\text{~for all $p\neq 0$}.
\end{align*}
In particular, $G'/G''\cong H_1(X,\Z G^{\ab})=0$, and 
the second claim is proved.
\end{proof}

\begin{example}
\label{ex:heisenberg1}
Let $H_{\R}$ be the algebraic group of $3\times 3$ unipotent matrices 
with entries in $\R$, and let $H_{\Z}$ be subgroup of integral matrices. 
The quotient space, $M=H_{\R}/H_{\Z}$ is a closed, orientable, 
aspherical  $3$-manifold, known as the Heisenberg 
nilmanifold.  Its fundamental group $G$ is the 
free, $2$-step nilpotent group of rank $2$.  Since $G$ is a
Poincar\'e duality group of dimension $3$, it cannot be an abelian duality
group.  We will return to this group in Example~\ref{ex:flat}.
\end{example}

\subsection{Abelian duality and extensions}\label{subsec:extensions}
In \cite[Thm.~3.5]{BE}, Bieri and Eckmann showed that the class of 
duality groups is closed under extensions.  In this section, we note
that some analogous results hold for abelian duality groups and spaces. 

Suppose $F\to E\to B$ is a fibration sequence of path-connected
CW-complexes.  The exact sequence of low-degree terms in the 
Serre spectral sequence reads as follows:
\begin{equation}
\label{eq:5term}
\xymatrixcolsep{16pt}
\xymatrix{
H_2(E,\Z)\ar[r] & H_2(B,\Z)\ar[r]^(.45){d} & H_1(F,\Z)_Q\ar[r] & 
H_1(E,\Z)\ar[r] & H_1(B,\Z)\ar[r] & 0
},
\end{equation}
where $Q=\pi_1(B)$.

\begin{definition}
\label{def:goodH1}
We say that the sequence $F\to E\to B$  is {\em $\ab$-exact}\/ if 
\begin{enumerate}
\item \label{a1} $Q$ acts trivially on $H_1(F,\Z)$; and
\item \label{a2} the map $d\colon H_2(B,\Z)\to H_1(F,\Z)$ 
is zero.
\end{enumerate}

In the presence of condition \eqref{a1}, condition \eqref{a2} 
is equivalent to the exactness of the sequence
$0\to N^{\ab}\to G^{\ab}\to Q^{\ab}\to  0$,
where $N=\pi_1(F)$ and $G=\pi_1(E)$.

\end{definition}
\begin{remark}\label{rem:inf.cyclic.covers}
For fibrations over $S^1$, the second condition follows from the first,
since the map $d$ in \eqref{eq:5term} is zero for dimensional reasons.  In
this case, $\ab$-exactness is equivalent to trivial monodromy action of
the base on the first homology group of the fiber.
\end{remark}

Although the definition is clearly quite restrictive, we will see that some
interesting group extensions in \S\ref{sect:hyparr} and \S\ref{sect:raag}
are indeed $\ab$-exact.

\begin{proposition}
\label{prop:extensions}
Suppose $F\to E\to B$ is an $\ab$-exact fibration of path-connected, 
finite-type CW-complexes.
\begin{enumerate}
\item \label{prop:extn:1} If 
$F$ and $B$ are $\ab$-duality spaces of dimensions $n$ and $r$,
respectively, then $E$ is an $\ab$-duality space of dimension $n+r$.
\item \label{prop:extn:2} 
If $F$ and $E$ are $\ab$-duality spaces of dimensions $n$ and $n+r$,
respectively, and $\dim B = r$, then $B$ is an $\ab$-duality 
space of dimension $r$.
\item \label{prop:extn:3} 
If $E$ and $B$ are $\ab$-duality spaces of dimensions $n+r$ and $r$,
respectively, and $\dim F = n$, then $F$ is an $\ab$-duality 
space of dimension $n$.
\end{enumerate}
\end{proposition}

\begin{proof}
Set $Q=\pi_1(B)$, $G=\pi_1(E)$, and $N=\pi_1(F)$.  The Serre 
spectral sequence has
\[
E_2^{pq}=H^p(B,H^q(F,\Z[G^{\ab}]))\Rightarrow H^{p+q}(E,\Z[G^{\ab}]).
\]
Clearly, the action of $N=\pi_1(F)$ on $\Z[G^{\ab}]$ factors through $N^{\ab}$. 
Furthermore, exactness of \eqref{eq:5term} insures that 
$\Z[G^{\ab}]\cong \Z[Q^{\ab}]\otimes_\Z \Z[N^{\ab}]$,
as a $\Z[N^{\ab}]$-module.
By hypothesis, $F$ is of finite-type, so $H^*(F,-)$ commutes with colimits, 
giving
\begin{equation}
\label{eq:extn:E2}
E_2^{pq}\cong H^p(B,\Z[Q^{\ab}]\otimes H^q(F,\Z[N^{\ab}])).
\end{equation}

To prove \eqref{prop:extn:1}, assume $F$ and $B$ are abelian duality spaces 
with dualizing modules $D_F$ and $D_B$, respectively.  
Then $E_2^{pq}=0$ unless $q=n$, in which case
\[
E_2^{pn}\cong H^p(B,\Z[Q^{\ab}]\otimes_\Z D_F).
\]
By Definition~\ref{def:goodH1}\eqref{a1}, the action of $Q$ on $\Z[N^{\ab}]$ is
trivial, so $D_F$ is also trivial as a $Q$-module.  Again, $B$ is of finite-type, so 
\[
E_2^{pn}\cong H^p(B,\Z[Q^{\ab}])\otimes_\Z D_F,
\]
which is zero unless $p=r$, where $E_2^{rn}\cong D_B\otimes_\Z D_F$,
and the spectral sequence gives
$H^p(E,\Z[G^{\ab}])\cong D_B\otimes_{\Z} D_F$ for $p=r+n$, and zero 
otherwise.  This is torsion-free, since $D_B$ and $D_F$ are torsion-free.
Moreover, $\dim E\leq \dim F+\dim B=r+n$, 
so $\dim E=r+n$. It follows that $E$ is an abelian duality space 
of dimension $r+n$.

To prove \eqref{prop:extn:2}, suppose now $F$ and $E$ 
are abelian duality spaces.  The same argument as above shows
\[
H^{p+n}(E,\Z[G^{\ab}])\cong H^p(B,\Z[Q^{\ab}])\otimes_\Z D_F,
\]
for all $p$, so $H^r(B,\Z[Q^{\ab}])\otimes_\Z D_F\cong D_E$.  Since
$D_F$ and $D_E$ are torsion-free, it follows that $H^p(B,\Z[Q^{\ab}])$ is
torsion-free, nonzero, and concentrated in dimension $p=r$. 
By hypothesis, $\dim B=r$, so $B$ is an abelian duality space.

To prove \eqref{prop:extn:3}, suppose $E$ and $B$ are abelian duality spaces.
From \eqref{eq:extn:E2}, since $B$ is of finite-type, we obtain
\[
E_2^{rq}\cong D_B\otimes H^q(F,\Z[N^{\ab}]),
\]
and $E_2^{pq}=0$ for $p\neq r$.  So the spectral sequence degenerates again
to give
\[
H^{n+q}(E,\Z[G^{\ab}])\cong D_B\otimes_\Z H^q(F,\Z[N^{\ab}]),
\]
and the argument concludes as in case \eqref{prop:extn:2}.
\end{proof}

An analogous statement holds for groups in place of spaces: we 
simply replace the hypothesis ``finite-type CW-complex of dimension $n$'' 
with ``$\FP$ group of cohomological dimension $n$.''

\subsection{Discussion and examples}
\label{subsec:disc}

We say a group $G$ is an almost-direct product of free groups, 
if it can be expressed as an iterated semidirect product of the 
form $G=F_{n_d}\rtimes \cdots \rtimes F_{n_1}$, 
where each group $F_{n_p}$ is free (of rank $n_p\ge 1$), 
and acts on each group $F_{n_q}$ with $q>p$ via an 
automorphism inducing the identity in homology.

Proposition \ref{prop:extensions}, part \eqref{prop:extn:1} 
has an immediate application. 

\begin{corollary}
\label{cor:iaext}
If $G=F_{n_d}\rtimes \cdots \rtimes F_{n_1}$ is an almost-direct 
product of free groups, then $G$ is an abelian duality group 
of dimension $d$. 
\end{corollary}

\begin{example}
\label{ex:iaext}
Both the pure braid group $P_n$ and the group of pure 
symmetric automorphisms $P\Sigma^{+}_n$ are almost 
direct products of the form $F_{n-1}\rtimes \cdots \rtimes F_{1}$.  
Thus they are abelian duality groups of dimension $n-1$.
\end{example}

The two conditions from Definition \ref{def:goodH1} are 
necessary for Proposition \ref{prop:extensions} to hold. 

\begin{example}
\label{ex:flat}
Let $G$ be the (discrete) Heisenberg group, which we saw in
Example~\ref{ex:heisenberg1} was not an abelian duality group.
Since $G$ is a split extension of $\Z$ by $\Z^2$, with monodromy 
$\left(\begin{smallmatrix} 1&1\\0&1\end{smallmatrix}\right)$, 
we see that condition \eqref{a1} is necessary.
Moreover, $G$ is also a central extension of $\Z^2$ by $\Z$, 
yet the map $d\colon H_2(\Z^2,\Z)\to H_1(\Z,\Z)$ 
is an isomorphism, thus showing that condition 
\eqref{a2} is necessary. 
\end{example}

\section{Abelian duality, jump loci, and propagation}
\label{sect:cv prop}

In this section, we note some qualitative constraints that the abelian
duality property imposes on a space.  We begin by recalling the
definitions of the Alexander invariants and jump loci associated 
with a space.

\subsection{Alexander invariants}
\label{subsec:alexinv}

As usual, we let $X$ be a connected, finite-type CW-complex, and 
$G=\pi_1(X,*)$ its fundamental group.  Moreover, we let $G^{\ab}$ 
be the abelianization of $G$, and $R=\Z{G^{\ab}}$.  

Suppose $X^{\nu}\to X$ is a regular abelian cover, classified by an 
epimorphism $G \xrightarrow{\ab} G^{\ab} \xrightarrow{\nu} H$, 
where $H$ is a (finitely generated) abelian group.   
By Shapiro's Lemma, $H_p(X^{\nu},\Z) \cong H_p(X, \Z{H})$, 
where $\Z{H}$ is viewed as an $R$-module via the 
map $\nu$. 

\begin{proposition}
\label{prop:shapiro}
If $X$ is an abelian duality space of dimension $n$, 
with dualizing module $D=H^n(X,\Z{G^{\ab}})$, then 
\begin{equation}
\label{eq:hpnu}
H_{\hdot}(X^{\nu},\Z) \cong \Ext_R^{n-\hdot}(D,\Z{H}).
\end{equation}
\end{proposition}

\begin{proof}
The claim follows at once from Proposition~\ref{prop:duality}.
\end{proof}

Consider now the universal abelian cover, $X^{\ab}\to X$.  
By definition, the $p$-th {\em Alexander invariant}\/ of $X$
is the $R$-module $H_p(X^{\ab},\Z)=H_p(X,R)$. 

\begin{corollary}
\label{cor:alexinv}
Suppose $X$ is an abelian duality space and $G^{\ab}$ is torsion-free.
If the $p$-th Alexander invariant is non-zero, then its 
support has dimension 
at most $p+m-n$, where $m$ is the rank of $G^{\ab}$.
\end{corollary}

\begin{proof}
By hypothesis, $R$ is a domain, so
$\Ext_R^q(D,R)$ is supported in codimension $q$ or 
higher, cf.~\cite[Thm.~1.1]{EHV92}.
On the other hand, the dimension of $R$ equals $m$.  The desired 
conclusion follows.
\end{proof}

\subsection{Jump loci of spaces}
\label{subsec:jump_loci-spaces}
Let $\widehat{G}:=\Hom(G,\k^*)$ denote the group
of multiplicative, $\k$-valued characters of $G$.  
We will identity $\widehat{G}=\widehat{G^{\ab}}$ 
with $\mspec R$.  

\begin{definition}
\label{def:cv}
The {\em characteristic varieties}\/ of $X$ are defined as
\begin{equation}
\label{eq:cvxdown}
\V_{i,d}(X;\k) := \V_{i,d}(C_\hdot(\widetilde{X}^{\ab}, \k)),
\end{equation}
where $C_\hdot(\widetilde{X}^{\ab},\k)$ denotes the equivariant cellular
chain complex of the universal abelian cover (with coefficients 
in $\k$), viewed as a chain complex over the group-ring $R=\k[G^{\ab}]$. 

The cohomological variant is defined analogously: 
\begin{equation}
\label{eq:cvxup}
\V^{i,d}(X;\k) := \V^{i,d}(C^\hdot(\widetilde{X}^{\ab}, \k)),
\end{equation}
where $C^\hdot(\widetilde{X}^{\ab}, \k)=C_\hdot(\widetilde{X}^{\ab}, \k)^\vee$ 
is the equivariant cochain complex of $\widetilde{X}^{\ab}$.

We note that the support loci of the Alexander invariants of 
$X$ are sometimes known as the Alexander varieties of $X$.  We let
\begin{align}
\label{eq:wix}
\W_{i,d}(X;\k):=&\W_{i,d}(C_\hdot(\widetilde{X}^{\ab},\k))\\
\notag
=&\supp\bigwedge\nolimits^d H_i(X,R).
\end{align}

The varieties $\W^{i,d}(X;\k)$ are defined analogously.
\end{definition}

\begin{remark}
By Proposition~\ref{prop:UCT},
$\V^{i,d}(X;\k)=\iota(\V_{i,d}(X;\k))$, for all integers $i$ and $d$, 
as noted in \cite[Lem.~4.3]{KP14}.  On the other hand, we remark that there is 
no such tight relationship between $\W_i(X;\k)$ and $\W^i(X;\k)$; we 
will come back to this point in Example~\ref{ex:CJL_vs_support}.
\end{remark}

\subsection{Propagation of characteristic varieties}
\label{subsec:cv propagate}
Perhaps the most striking consequence of abelian duality is the following
nestedness property of the characteristic varieties.  Let $\k$ be an 
algebraically closed field.

\begin{theorem}
\label{th:prop}
Suppose $X$ is an abelian duality space of dimension $n$ over $\k$.
Then the characteristic varieties of $X$ propagate: that is,
for any character $\rho\in\widehat{G}$, if $H^p(X,\k_\rho)\neq 0$, 
then $H^q(X,\k_\rho)\neq0$ for all $p\leq q\leq n$.
Equivalently, 
\[
\set{\one}=\V^0(X,\k)\subseteq \V^1(X,\k)\subseteq\cdots\subseteq\V^{n}(X,\k).
\]
\end{theorem}

\begin{proof} 
Since $X$ is an abelian duality space,
the equivariant cochain complex $C^\hdot(\tilde{X}^{\ab},\k)$
satisfies the hypotheses of Proposition~\ref{prop:general_propagation}.
\end{proof}

By considering the trivial representation, we note in particular:
\begin{corollary} 
\label{cor:betti numbers}
Suppose $X$ is an abelian duality space of dimension 
$n\ge 1$.  Then  $b_p(X)>0$, for all $0\le p \le n$. 
\end{corollary}

\begin{remark}
\label{rem:KT}
In \cite{Ha}, J.-C.~Hausmann showed that an analogue of the 
Kan--Thurston Theorem holds for duality groups: given any 
finite CW-complex $X$, there is a duality group $G$ and a 
map $BG\to X$ inducing an isomorphism in homology. 
In view of Corollary \ref{cor:betti numbers}, no such result 
holds for abelian duality groups. 
\end{remark}

\begin{example}
\label{ex:surfaces}
Let $S_g$ be a surface of genus $g>1$.  Let $\Pi_g=\pi_1(S_g)$.  
It is well-known that $S_g\simeq K(\Pi_g,1)$, and $\Pi_g$ is a 
$2$-dimensional Poincar\'{e} duality group.  That is,
$H^p(S_g,\Z{\Pi}_g)$ is isomorphic to $\Z$ if $p=2$, 
and is $0$ otherwise. 
One can check, though, that for all $\rho\in \Hom(\Pi_g,\k^*)$ 
except the trivial representation, we have
$H^1(S_g,\k_\rho)\cong\k^{2g-2}$ and $H^2(S_g,\k_\rho)=0$.
By Theorem~\ref{th:prop}, then, the surface of genus $g>1$ 
is {\em not}\/ an abelian duality space.  This example also shows that
propagation of characteristic varieties fails to hold, in general,
for duality spaces.
\end{example}

We also note that the rank of $G^{\ab}$ must be at least the dimension of
$X$.  More precisely,

\begin{proposition}
\label{prop:leastdim}
If $X$ is an abelian duality space of dimension $n$ over $\k$ and
$\ch \k$ does not divide the order of the torsion subgroup of $H_1(X,\Z)$,
then $\dim_\k H^1(X,\k) \ge n$.
\end{proposition}

\begin{proof}
Set $m= \dim_\k H^1(X,\k)$. Recall that $R=\k[G^{\ab}]$, and 
the cochain complex of the universal abelian cover is a free 
resolution of the dualizing module $D$ as an $R$-module.  First suppose
$G^{\ab}$ is torsion-free.  Then 
$\pdim_R D\leq n$.  On the other hand, by Proposition~\ref{prop:duality},
$\Tor^R_n(D,\k)=H^0(X,\k)\neq0$, where $\k$ is the trivial module.
By our hypothesis on $\ch\k$, $R$ is a regular $\k$-algebra of (global)
dimension $m$, from which it follows that $\pdim_R D = n\leq m$.
\end{proof}

Along the same lines, the dimensions of irreducible components of $\V^i(X,\k)$ 
are not arbitrary.  As the next proposition shows, such components 
correspond to associated primes of the dualizing module $D$.  

\begin{proposition} 
\label{prop:dimcomp}
Let $X$ be an abelian duality space of dimension $n$.  Suppose $G^{\ab}$ has
rank $m$, and again that it has no torsion elements of order $\ch\k$.
If $Z=V(P)$ is an irreducible component of $\V^i(X,\k)$ for a prime ideal $P$,
then $\pdim_R R/P\geq n-i$.  If, moreover, $Z$ is Cohen--Macaulay, then
$\dim Z\leq i+n-m$.
\end{proposition}

\begin{proof}
As in the proof of \cite[Theorem~4.1(c)]{EPY03}, we let $Z$ be such a 
component of dimension $d$, and let $\rho$ be the generic point of $Z$.
Then $H^i(X,\kappa(\rho))\neq0$, so by Lemma~\ref{lem:local} and 
Proposition~\ref{prop:duality}, 
$\Tor^R_{n-i}(D,R/P)\neq0$.  This implies $n-i\leq \pdim R/P$.  If
$Z$ is Cohen--Macaulay, then $\dim Z+\pdim_R P = m$, from which the second
claim follows. 
\end{proof}
\begin{remark}
Though we must 
impose a condition on the characteristic of $\k$ if $H_1(X,\Z)$ has
a nontrivial torsion subgroup, we do not know of an example of an abelian
duality space for which $H_1(X,\Z)$ is {\em not} torsion-free.
\end{remark}

\subsection{Propagation of resonance varieties}
\label{subsec:res-spaces}
In the previous section, we saw that the abelian duality property
informed on the behavior of the characteristic varieties.  Here, we
note that the situation is analogous for the EPY property and 
resonance varieties.

First, we define the resonance varieties of a space.
As before, let $X$ be a connected, finite-type CW-complex. 
Consider the cohomology algebra $A=H^* (X,\k)$. 
If $\ch\k =2$, we will assume that $H_1(X,\Z)$ has no $2$-torsion.  
In this case, it is readily checked that $a^2=0$ for every $a\in A^1$; 
thus, $A$ can be viewed, in a natural way, as a module over 
the exterior algebra $E=\bigwedge A^1$.

\begin{definition}
\label{def:resvar}
The {\em resonance varieties}\/ of $X$ (over $\k$) are the 
resonance varieties of $A$, viewed as a module over $E$:
\begin{equation}
\label{eq:rvx}
\RR^{i,d}(X,\k)=\{a \in H^1(X,\k) \mid 
\dim_{\k} H^i(A,a\cdot ) \ge  d\}.
\end{equation}

For simplicity, we shall write $\RR^i(X,\k)=\RR^{i,1}(X,\k)$. 
For a group $G$ with finite-type classifying space, 
we shall also write $\RR^{i,d}(G,\k)=\RR^{i,d}(K(G,1),\k)$.
\end{definition}

We begin with the analogous result to Theorem~\ref{th:prop}, 
which is simply a restatement of Theorem~\ref{thm:resprop}, applied to
the cohomology ring of a space.

\begin{theorem}
\label{thm:res_prop}
Suppose $X$ is a space of dimension $n$ 
with the EPY property over a field $\k$.  Then
the resonance varieties of $X$ propagate: that is, for any $a\in A^1$,
if $H^p(A,a\cdot)\neq0$, then $H^q(A,a\cdot)\neq0$ for all $p\leq q\leq n$,
where $A=H^\hdot(X,\k)$.  Equivalently,
\[
\set{0}=\RR^0(X,\k)\subseteq \RR^1(X,\k)\subseteq\cdots\subseteq\RR^{n}(X,\k).
\]
\end{theorem}

We continue with a further examination of the interplay between the
duality properties of a space and the nature of its resonance varieties.
We start with a simple observation relating Poincar\'e duality to the top 
resonance variety.

\begin{proposition}
\label{prop:res pd}
Let $M$ be a compact, connected, orientable manifold 
of dimension $n$.  Then $\RR^n(M,\k)=\{0\}$. 
\end{proposition}

\begin{proof}
Let $\mu \in H^n(M,\Z)\cong \Z$ be the generator defining 
the orientation on $M$, and $\mu_\k$ its image in $H^n(M,\k)\cong \k$. 
Given any $a\in H^1(M,\k)$, Poincar\'e duality guarantees the 
existence of a cohomology class $b\in H^{n-1}(M,\k)$ such that 
$a\cup b= \mu_\k$, and we are done.
\end{proof}

The same argument proves the following:  If $G$ is a Poincar\'e 
duality group of dimension $n$, then $\RR^n(G,\k)=0$.  In 
dimension $n=3$, we can say a bit more.  

\begin{proposition}[\cite{DS09}, Prop.~5.1]
\label{prop:res3}
Let $M$ be a closed, orientable $3$-manifold.  
If $b_1(M)$ is even, then $\RR^1(M,\C)=H^1(M,\C)$.
\end{proposition}

Putting together Propositions \ref{prop:res pd} and \ref{prop:res3}, 
we obtain the following corollary.

\begin{corollary}
\label{cor:resprop3}
Let $M$ be a closed, orientable $3$-manifold.  
If $b_1(M)$ is even and non-zero, then the resonance varieties of $M$ 
do not propagate (in characteristic $0$). 
\end{corollary}

We shall see a concrete instance of this phenomenon in 
Example \ref{ex:heisenberg} below.

\subsection{Minimal complexes}
\label{subsec:koszul ab}

Here, we find that, if $X$ is a minimal CW-complex, then the abelian 
duality and EPY properties are related.

We say the CW-structure on $X$ is {\em minimal}\/ 
if the number of $i$-cells of $X$  equals the 
Betti number $b_i(X)$, for every $i\ge 0$. 
Equivalently, the boundary maps in the cellular 
chain complex $C_{\hdot}(X,\Z)$ are all zero maps. 
In particular, the homology groups $H_i(X,\Z)$ are all 
torsion-free, and thus resonance varieties are defined 
in all characteristics.

For instance, if $M$ is a smooth, closed manifold admitting 
a perfect Morse function, then $M$ has a minimal cell structure. 
Evidently, spheres, tori, and orientable surfaces, as well as 
products thereof are of this type.

\begin{theorem}[\cite{PS10}, Thm.~12.6]
\label{thm:linaom}
Let $X$ be a minimal CW-complex. Then the linearization 
of the cochain complex $C^{\hdot}(X^{\ab},\k)$ 
coincides with the universal Aomoto complex of 
$A=H^*(X,\k)$.
\end{theorem}

To see how this works concretely, 
pick an isomorphism $H_1(X,\Z)\cong\Z^m$, and 
identify the group ring $\k[\Z^m]$ with  the Laurent polynomial ring 
$\Lambda=\k[t_1^{\pm 1}, \dots , t_m^{\pm 1}]$.  
Next, filter $\Lambda$ by powers of the maximal ideal 
$I=(t_1-1,\dots,t_m-1)$, and identify the associated graded 
ring, $\gr(\Lambda)$, with the polynomial ring 
$S=\k[x_1,\dots,x_m]$, via the ring map $t_i-1\mapsto x_i$.  

The minimality hypothesis allows us to identify 
$C_{i} (X^{\ab}, \k)$ with $\Lambda \otimes_{\k} H_i(X,\k)$   
and $C^{i} (X^{\ab}, \k)$ with $A^{i} \otimes_{\k} \Lambda$.  
Under these identifications, the boundary map 
$\partial_{i+1}^{\ab}\colon C_{i+1} (X^{\ab}, \k)
\to C_{i} (X^{\ab}, \k)$ dualizes to a map $\delta^i\colon 
A^i \otimes_{\k} \Lambda \to A^{i+1} \otimes_{\k} \Lambda$. 
Let $\gr(\delta^i)\colon 
A^i \otimes_{\k} S \to A^{i+1} \otimes_{\k} S$ 
be the associated graded of $\delta^i$, and let 
$\gr(\delta^i)^{\linn}$ be its linear part.   
Theorem \ref{thm:linaom} then provides an identification
with $\bL(A)$ from \eqref{eq:lp},
\begin{equation}
\label{eq:linaom}
\gr(\delta^i)^{\linn}=d^i\colon 
A^i \otimes_{\k} S \longrightarrow A^{i+1} \otimes_{\k} S.
\end{equation}

\begin{theorem}
\label{thm:kosab}
If $X$ is a minimal CW-complex, and the $E$-module  $A=H^*(X,\k)$ has the 
EPY property, then the $I$-adic completion of $H^p(X,\k{G^{\ab}})$ vanishes 
for $p\ne n$, where $n$ is the socle degree of $A$.
\end{theorem}

\begin{proof}
By our minimality assumption, the equivariant 
spectral sequence from \cite{PS10} starts at 
\begin{align}
\label{eq:e0pq}
E_0^{pq} &= \gr_p (C^{p+q} (X^{\ab} ,\k) ) \\  \notag
& \cong  A^{p+q} \otimes S_p.\notag
\end{align}
Furthermore, the differential $d_0$ vanishes, and so $E_1^{pq} = E_0^{pq}$, 
with differential $d_1$ given by \eqref{eq:linaom}.  Hence, 
$E_2^{pq} = H^{p+q} (A\otimes S, \delta )_p$.  
By our Koszulness assumption, 
\begin{equation}
E_2^{pq} = \begin{cases}
F(A)_p & \text{if $p+q=n$},\\
0 & \text{otherwise}. 
\end{cases}
\end{equation}

In particular, $d_2=0$, and so $E_{\infty}^{pq} =E_2^{pq} $. 
Finally, we also know from \cite{PS10} that the spectral sequence 
converges to the $I$-adic completion of $H^*(X,\k{H})$. 
The desired conclusion readily follows.
\end{proof}

\begin{question}
In principle, there is no reason to expect that the conclusion of 
Theorem~\ref{thm:kosab} would hold for the (uncompleted) modules 
$H^p(X,\k{G^{\ab}})$.  However, we have no example.  If a minimal
CW-complex $X$ has the EPY property, must $X$ be an abelian duality
space?
\end{question}

\subsection{Discussion and example}
\label{subset:discuss}

We conclude this section with several more examples.  The first is 
a (non-formal) space for which the characteristic varieties propagate, 
but the resonance varieties do not.

\begin{example}
\label{ex:heisenberg}
Once again, we consider the Heisenberg manifold of Example~\ref{ex:heisenberg1}.
It is readily checked that  $\V^i(M)=\{1\}$ for all $i\le 3$; thus, 
the characteristic varieties of $M$ propagate. On the other hand, 
$H^1(M,\Z)=\Z^2$, and the cup product map $H^1\otimes H^1 \to H^2$ 
vanishes. Therefore, resonance does not propagate: 
$\RR^1(M,\k)=\k^2$, yet $\RR^3(M,\k)=\{0\}$.

Finally, note that $M$ is an $S^1$-bundle over $T^2$ with Euler 
number $1$, and thus admits a minimal cell decomposition.  
Hence, by Theorems \ref{thm:resprop} and \ref{thm:linaom},
we see again that the manifold $M$ is 
{\em not}\/ an abelian duality space.  This, despite the fact that 
$M^{\ab}$ is homotopy equivalent to $S^1$, which of course is 
a $\PD_1$-space. 
\end{example}

In the case of $2$-complexes with non-negative Euler characteristic,
propagation occurs because of dimensional considerations.  However,
such spaces need not be abelian duality spaces, as we see below.

\begin{proposition}
\label{prop:chi0}
Suppose $X$ is a connected, finite $2$-complex with $\chi(X)\geq0$.  
If $H^1(X,\k)\neq0$, the resonance and characteristic varieties of 
$X$ both propagate.
\end{proposition}

\begin{proof}
Clearly $\V^0(X)=\set{{\mathbbm 1}}$, and ${\mathbbm 1}\in\V^i(X)$ for
$i=1,2$, by hypothesis.
Now consider a character $\rho\neq{\mathbbm 1}$.  The universal cover of 
$X$ has a cellular chain complex
\[
\xymatrix{
C_\hdot(\widetilde{X}):=\:  \k[G]^{c_2}\ar[r] & \k[G]^{c_1}\ar[r] & \k[G]^{c_0},
}
\]
for some integers $c_i$ with $\chi(X)=c_2-c_1+c_0\geq 0$.
Then $H^\hdot(X,\k_{\rho})$ is computed by the complex
\[
\xymatrix{\Hom_{\k[G]}(C_\hdot(\widetilde{X},\k_{\rho}))\colon \:
\k^{c_2} & \k^{c_1}\ar[l] & \k^{c_0}\ar[l].
}.
\]
Since $H^0(X,\k_\rho)=0$, we have $c_2-c_1\geq 0$, which means if
$\rho\in \V^1(X,\k)$, then $\rho\in \V^2(X,\k)$.
The analogous argument applies for the resonance varieties as well.
\end{proof}

Finally, here is an example of a group which isn't an abelian 
duality group, yet for which resonance does propagate.

\begin{example} 
\label{ex:prop min nondual}
Let $X$ be the presentation $2$-complex for the $1$-relator group 
\[
G=\langle x_1, x_2 \mid [x_1,x_2] [x_2, [x_2, x_1]] =1\rangle.
\]
Since the relator is 
not a proper power, $X$ is a $K(G,1)$.  A standard computation 
shows that $H^{\hdot}(X,\Z)=H^{\hdot}(T^2,\Z)$; in particular, $G$ is a 
$\PD_2$-group of nonnegative Euler characteristic, 
so the proposition above applies to show 
that its resonance and characteristic varieties propagate.  Indeed,
$\RR^1(X,\k)=\RR^2(X,\k)=\{0\}$, and moreover,
\[
\V^1(X,\k)= \V^2(X,\k) =
\{ (t_1,t_2)\in (\k^*)^2 \mid t_2 = 2 \}\cup \{1\}.
\]

Next we note that the Euler characteristic assumption is necessary.
Let $Y=X\vee S^1$.   Again, $X$ is an aspherical, minimal 
$2$-complex.  Moreover,  both $\RR^1(Y,\k)$ and $\RR^2(Y,\k)$ 
are equal to $H^1(Y,\k)=\k^3$, and so resonance propagates 
for $Y$.  On the other hand, $\V^1(Y,\k)=H^1(Y,\k^*)=(\k^*)^3$, yet 
\[
\V^2(Y,\k) =
\{ (t_1,t_2,t_3)\in (\k^*)^3 \mid t_2 = 2 \}\cup \{1\}.
\]
Therefore, if $\ch(\k)\ne 2$, the characteristic varieties of $Y$ do not 
propagate.  By Theorem~\ref{th:prop}, the space $Y$ is not 
an abelian duality space.
\end{example}

\section{Hyperplane arrangements}
\label{sect:hyparr}
In this section, we give some examples of abelian duality spaces.

\subsection{Linear arrangements}
\label{subsec:linear}
Let $\A$ be an essential, central hyperplane 
arrangement in $\C^{n+1}$.  Let 
$M(\A)=\C^{n+1}\setminus \bigcup_{H\in\A}H$ be the complement of the 
arrangement, and $U(\A)=\PP^n\setminus\bigcup_{H\in\A}\overline{H}$ the
projective complement.  Then $M(\A)\cong U(\A)\times \C^*$, and $U(\A)$
may be regarded as a the complement of a possibly non-central arrangement in 
$\C^{n-1}$ by a choice of a hyperplane at infinity.  Without loss of 
generality, then, we mostly restrict our attention to projective complements.
$U(\A)$ is a Stein manifold, and thus has the 
homotopy type of a connected CW-complex of dimension $n$.  Moreover, 
as shown in \cite{DP03, Ra02}, the cell structure can be chosen to be minimal. 

The cohomology ring $H^\hdot(U(\A),\Z)$ was computed 
by Brieskorn in the early 1970s, building on pioneering 
work of Arnol'd on the cohomology ring of the braid arrangement.  
It follows from Brieskorn's work that the space $U(\A)$ is formal.
We let $A(\A):=H^\hdot(U(\A),\Z)$, the (projective) 
Orlik-Solomon algebra of $\A$.

As shown in \cite{DJLO11}, the complement 
$U(\A)$ is a duality space.   It turns out that a hyperplane 
complement is also an abelian duality space.  Both of these
properties follow from a more general cohomological vanishing
result developed in \cite{DSY16}, an instance of which is the following:

\begin{theorem}[\cite{DSY16}, Thm.~5.6]
\label{thm:arr vanish}
Let $U(\A)$ be the complement of an essential arrangement in $\PP^n$, 
and set $H=H_1(U(\A), \Z)$.  Then 
$H^p(U(\A),\Z{H})=0$ for all $p\neq n$, and 
$D(\A):=H^n(U(\A),\Z{H})$ is free abelian.
\end{theorem}

A main result of \cite{EPY03} is that the shifted cohomology ring $A(\A)(n)$
is a Koszul module over the exterior algebra $E=\bigwedge A^1(\A)$: that is,
it satisfies the EPY property.  So arrangement complements satisfy the 
hypotheses of both Theorems \ref{th:prop} and \ref{thm:res_prop}.

\begin{corollary}
\label{cor:proparr}
Let $\A$ be a hyperplane arrangement of rank $n+1$.  Then both the 
characteristic and the resonance varieties of its projective 
complement, $U=U(\A)$, propagate:
\[
\V^1(U,\k)\subseteq\cdots\subseteq\V^{n}(U,\k)
\quad\text{and}\quad 
\RR^1(U,\k)\subseteq\cdots\subseteq\RR^{n}(U,\k).
\]
\end{corollary}

\subsection{Examples and discussion}
\label{subsec:ex}

We now give an example of an arrangement for which the support loci 
of the Orlik--Solomon algebra do {\em not}\/ propagate.

\begin{example}
\label{ex:CJL_vs_support}
Let $\A$ denote the graphic arrangement defined by equations $x_i-x_j$
for edges $\set{i,j}$ in the graph below.
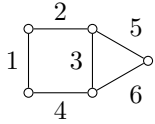
\[
\begin{tikzpicture}[scale=0.85,baseline=(current bounding box.center),
plain/.style={circle,draw,inner sep=1.2pt,fill=white}]
\node[plain] (1) at (0,0) {};
\node[plain] (2) at (0,1) {};
\node[plain] (3) at (1,1) {};
\node[plain] (4) at (1,0) {};
\node[plain] (5) at (1.87,0.5) {};
\draw (1) -- node [left] {1} (2) -- node[above] {2}  (3) --
node[left] {3} (4) -- node[below] {4} (1);
\draw (3) -- node[above right] {5} (5) -- node[below right] {6} (4);
\end{tikzpicture}
\]
If $\pi$ is a partition of $[6]$, we let $P_\pi$ denote the
codimension-$k$ linear subspace of $V$ given by 
equations $\sum_{j\in \pi_s}x_j=0$, for each $1\leq s\leq k$, where
$k=\abs{\pi}$.  

Let $A=H^\hdot(U(\A),\Q)$.  
Calculations as in \cite{Dn16} show that
$\W^1(\bL(A))=P_{1|2|3|456}$ and $\W^2(\bL(A))= P_{1234|5|6}$; we
illustrate Proposition~\ref{prop:VfromW_again} by noting that
\begin{align*}
\RR^0(A)=&\W_0(\bL(A))=P_{1|2|3|4|5|6},\\
\RR^1(A)=&\W_1(\bL(A)),\\
\RR^2(A)=&\W_1(\bL(A))\cup \W_2(\bL(A)),\quad\text{and}\\
\RR^3(A)=&\W_3(\bL(A))=P_{123456}.
\end{align*}

Since $\W^1(\bL(A))\not\subseteq \W^2(\bL(A))$, we see
in particular that, unlike the resonance varieties, the support 
loci of the complex $\bL(A)$ do not have the propagation property. 
\end{example}

\begin{remark}  
\label{rem:deep prop}
In \cite{Bu11, Bu14}, Budur establishes the following inclusions for the 
complement $M$ of an arrangement of rank $n$:
\begin{align*}
&\RR^i(M,\C) \subseteq \RR^i_2(M,\C)&& \text{for $i\le n-2$},\\
&\RR^i(M,\C) \subseteq \RR^i_d(M,\C)&& \text{for $i< n-2$ 
and $d< 1+ (n-3)/(i+1)$}.
\end{align*}
It would be interesting to see whether this type of `deeper' propagation 
arises in the more general setup of minimal, abelian duality spaces.
\end{remark}

\subsection{Elliptic arrangements}
\label{subsec:elliptic}
Let $E$ be an elliptic curve.  An elliptic arrangement in $E^{\times n}$
is a finite collection $\A=\{H_1,\dots, H_m\}$ of fibers of group homomorphisms 
$E^{\times n}\to E$, see \cite{LV12, Bi16}.  
Such homomorphisms are parameterized by 
integer vectors: writing $E$ as an additive group, 
we may write $H_i=f_i^{-1}(\zeta_i)$ for some point 
$\zeta_i\in E$, where
\begin{equation}
\label{eq:ellpoly}
f_i(x_1,\ldots,x_n)=\sum_{j=1}^n a_{ij} x_j,
\end{equation}
and $A=(a_{ij})$ is a $m\times n$ integer matrix. 
We let $\corank(\A):=n-\rank(A)$.

Let $U(\A) = E^{\times n}\setminus \bigcup_{1=1}^m H_i$ 
be the complement of our elliptic arrangement. 
We show in \cite[Cor.~6.4]{DSY16} that  $U(\A)$ 
is both a duality and an abelian duality space of dimension $n+r$, 
where $r$ is the corank of $\A$.  Applying Theorem \ref{th:prop}, 
we obtain the following immediate corollary.

\begin{corollary}
\label{cor:prop_ellarr}
Let $\A$ be an elliptic arrangement in $E^{\times n}$.  Then the 
characteristic varieties of its complement, $U=U(\A)$, propagate:
\[
\V^1(U,\k)\subseteq\cdots\subseteq\V^{n}(U,\k).
\]
\end{corollary}

\begin{remark} 
\label{rem:res_ell}
We do not know whether complements of elliptic arrangements 
satisfy the EPY property, or whether the resonance varieties of 
elliptic arrangements propagate.
\end{remark}

\subsection{Milnor fibers of linear arrangements}
\label{subsec:mf}
Now let $f_{\A}$ be a reduced 
defining polynomial for $\bigcup_{H\in\A}H$, a product of linear forms.
The evaluation map $f_{\A}\colon \C^{n+1}\to \C$ restricts to a 
fibration, $F(\A)\to M(\A)\to \C^*$, known as the Milnor fibration.  For
a full discussion, see \cite{DS14, Su14b}.  The monodromy action of $\pi_1(\C^*)$ on 
$H_\hdot(F,\Z)$ is determined by the action of a generator,
$h_\hdot \colon H_\hdot(F,\Z)\to H_\hdot(F,\Z)$.  The order of 
$h_\hdot$ divides $N := \deg f_{\A}=\abs{\A}$.

\begin{theorem}
\label{thm:milnor}
Suppose that $\A$ is an essential, central arrangement of rank $n+1$, 
and the monodromy action of
$h_1$ on $H_1(F(\A),\Z)$ is trivial.  Then $F(\A)$ is an abelian duality space
of dimension $n$.
\end{theorem}

\begin{proof}
We know $M(\A)$ is an abelian duality space, by Theorem~\ref{thm:arr vanish}.
By Remark~\ref{rem:inf.cyclic.covers}, the fibration sequence 
$F(\A)\to M(\A)\to \C^*$ is $\ab$-exact provided that $h_1$ is the identity
map.  The conclusion then follows from Proposition \ref{prop:extensions}.
\end{proof}

\begin{remark}
The condition that $h_1$ acts nontrivially on $H_1(F,\Z)$ is interesting
but rather special: see \cite[Thm.\ 5.1]{Su14b} for a complete discussion.  
For most arrangements (in some sense), then, 
the Milnor fiber is an abelian duality space.

Nevertheless, even if the monodromy action of
$h_1$ on $H_1(F,\Z)$ is non-trivial, the Milnor fiber $F$ can still be an abelian 
duality space.  For instance, if $\A$ is a pencil of $3$ lines in $\C^2$, then 
the characteristic polynomial of $h_1$ is $(1-t)(1-t^3)$, yet 
$F(\A)=E \setminus \{1,\omega , \omega^2\}$, where $\omega^3=1$, 
and the claim follows from the discussion in \S\ref{subsec:elliptic}.
It would be interesting to know if the Milnor fiber of an arrangement
is an abelian duality space in general.

We note that, $h_p$ acts nontrivially on $H_p(F(\A),\Z)$ for 
{\em some} $p$ unless $\A$ is a decomposable arrangement and the number of
hyperplanes in each block have greatest common divisor $1$: see 
\cite{Di12}.  We also note 
that other authors consider the (non)triviality of $h_1$ over
$\C$ rather than $\Z$; however, we cannot rule out the possibility that
$H_1(F,\Z)$ has torsion: we refer to the discussion in \cite{DS14}.
\end{remark}

The propagation of characteristic varieties we observed for arrangement
complements in Corollary~\ref{cor:proparr} implies the following property
of monodromy eigenspaces of Milnor fibers.  If $\zeta\in\k^*$ is an 
eigenvalue of $h$ on $H_{\hdot}(F(\A),\k)$ for an algebraically closed field
$\k$ of characteristic relatively prime to $N$, let $H_p(F(\A),\k)_{\zeta}$
denote the corresponding eigenspace.

\begin{theorem}
\label{thm:mfprop}
If $\A$ is an arrangement of rank $n+1$ and 
$H_p(F(\A),\k)_{\zeta}\neq0$ for some $p<n$, then
$H_{q}(F(\A),\k)_{\zeta}\neq 0$ as well, for all eigenvalues $\zeta$, 
and all $p\le q \le n$.
\end{theorem}

\begin{proof}
Let $\delta\colon \pi_1(U(\A))\to \Z/N\Z$ 
denote the classifying map of the cover $F(\A)\to U(\A)$.  
If $\zeta$ is a primitive $N$th
root of unity, let $\k_i=\widehat{\delta}(\zeta^i)$.  Then 
we recall that
\begin{eqnarray*}
H_p(F(\A),\k)&=&\bigoplus_{i=0}^{N-1}H_p(U(\A),\k_i)\\
&=&\bigoplus_{i=0}^{N-1}H_p(F(\A),\k)_{\zeta^i}
\end{eqnarray*} for all $p$.  
The result then follows from Corollary~\ref{cor:proparr}.
\end{proof}

\section{Right-angled Artin groups and their relatives}
\label{sect:raag}

\subsection{Polyhedral products}
\label{subsec:gmac}
We start by recalling the construction of a `generalized moment-angle complex' 
(or, `polyhedral product'); for more details, we refer to \cite{DS07, BBCG07, Az13} 
and references therein. 

Let $(\X,\AA):=(X_i,A_i)_{1\leq i\leq m}$ be an $m$-tuple of connected, 
finite CW-pairs, and let $L$ be a simplicial complex on vertex set $\sV=[m]$.  
We then set 
\begin{equation}
\label{eq:zlx}
\ZZ_L(\X,\AA)=\bigcup_{\tau\in L}(\X,\AA)^\tau,
\end{equation}
where $(\X,\AA)^\tau=\prod_{i\in \sV} \X_{\tau, i}$, and 
\begin{equation}
\label{eq:xtau}
 \X_{\tau, i}= \begin{cases} X_i & \text{if $i\in\tau$;}\\
A_i & \text{if not.}\end{cases}
\end{equation}

Here, we will mainly be interested in the case when 
each $A_i=*=e^0$, a distinguished zero-cell in each $X_i$.
If, moreover, each $X_i=X$, we write $\ZZ_L(X,*)$ in place of $\ZZ_L(\X,*)$.

Now let $S^1=e^0\cup e^1$ be the circle, endowed with the standard 
cell decomposition.  The resulting polyhedral product, $T_L=\ZZ_L(S^1, e^0)$, 
is a subcomplex of the $m$-torus $T^m$.  
Thus, $T_L$ comes endowed with a minimal cell decomposition. 
The cohomology ring $H^\hdot(T_L,\k)$ is the 
exterior Stanley--Reisner ring $\k\langle L\rangle$, with generators 
the duals $v^*$, and relations the monomials corresponding 
to the missing faces of $L$. 

The fundamental group $G_L=\pi_1(T_L)$ is the 
{\em right-angled Artin group}\/ determined by the graph 
$\Gamma=L^{(1)}$, with presentation consisting of a generator 
$v$ for each vertex $v$ in $\sV$, and a commutator relation 
$vw=wv$ for each edge $\{v,w\}$ in $\Gamma$.  
A classifying space for the group $G_{\Gamma}=G_L$ is the toric complex 
$T_{\Delta}$, where $\Delta=\Delta_{L}$ is the flag complex 
of $L$, i.e., the maximal simplicial complex with 
$1$-skeleton equal to the graph $\Gamma$. 

Both the characteristic and resonance varieties of a toric complex 
were computed in \cite{PS09}. To state those results, start by 
identifying $H_1(T_K,\Z)=\Z^m$, with generators 
indexed by the vertex set $\sV=[m]$.  This allows us to  
identify $H^1(T_L,\k)$ with the vector space $\k^{\sV}=\k^m$, 
and $H^1(T_L,\k^{\times})$ with the algebraic torus 
$(\k^{\times})^{\sV}=(\k^{\times})^m$.  For each 
subset $\sW \subseteq \sV$, let $\k^{\sW}$ be 
the respective coordinate subspace, and let 
$(\k^{\times})^{\sW}$ be the respective algebraic 
subtorus.  

\begin{theorem}[\cite{PS09}, Thms.~3.8 and 3.12]
\label{thm:cjl tc}
With notation as above,
\begin{equation*}
\label{eq:cjl tc}
\V^i_d(T_L)= \bigcup_{\sW} \, 
(\k^{\times})^{\sW}
\quad \text{and}  \quad
\RR^i_d(T_L,\k)= \bigcup_{\sW} \, 
\k^{\sW} 
\end{equation*}
where, in both cases, the union is taken over 
all subsets $\sW\subset\sV$ for which 
\[
\sum_{\sigma\in L_{\sV\setminus \sW}}
\dim_{\k} \widetilde{H}_{i-1-\abs{\sigma}} 
(\lk_{L_\sW}(\sigma),\k) \ge d.
\] 
\end{theorem}

In the above, $L_\sW$ denotes the simplicial subcomplex 
induced by $L$ on $\sW$, and $\lk_{K}(\sigma)$ denotes 
the link of a simplex $\sigma$ in a subcomplex $K\subseteq L$.  
In particular, 
\begin{equation}
\label{eq:cjl raag}
\V^1(G_{\Gamma},\k) = \bigcup_{\sW}\, (\k^{\times})^{\sW} 
\quad \text{and}  \quad
\RR^1(G_{\Gamma},\k) = \bigcup_{\sW}\,  \k^{\sW},
\end{equation}
where the union is taken over all maximal subsets 
$\sW\subset\sV$ for which the induced graph
$\Gamma_{\sW}$ is disconnected. 

\subsection{Cohen--Macaulay complexes}
\label{subsec:cm}
Recall that a $n$-dimensional simplicial complex $L$ is 
{\em Cohen--Macaulay}\/  if for each simplex $\sigma\in L$, 
the reduced cohomology $\widetilde{H}^\hdot(\lk(\sigma),\Z)$ 
is concentrated in degree $n -\abs{\sigma}$ and is torsion-free. 
The analogous definition can be made over any coefficient ring $\k$. 
For a fixed $\k$, the Cohen--Macaulayness of 
$L$ is a topological property: it depends only on the homeomorphism 
type of $L$.  For $\sigma=\emptyset$, the condition means that 
$\widetilde{H}^\hdot(L,\Z)$ is concentrated in degree $n$; it also implies 
that $L$ is pure, i.e., all its maximal simplices have dimension $n$. 

The Cohen--Macaulay property turns out to be equivalent to duality and
abelian duality, as we see in the next two results.

\begin{theorem}[\cite{BM01, JM05}]
\label{prop:cm raag}
A right-angled Artin group $G_{\Gamma}$ is a 
duality group if and only if the flag complex $\Delta_{\Gamma}$ 
is Cohen--Macaulay.  Moreover, $G_{\Gamma}$ is a Poincar\'e 
duality group if and only if $\Gamma$ is a complete graph. 
\end{theorem}

\begin{theorem}
\label{thm:cm toric}
Let $L$ be a $n$-dimensional complex.  Then $L$ is Cohen--Macaulay 
over $\k$ if and only if the toric complex $T_L$ is an abelian duality 
space (of dimension $n+1$).
\end{theorem}

\begin{corollary}
\label{cor:abdual raag}
A right-angled Artin group $G_\Gamma$ is an abelian duality group if and only
if the flag complex $\Delta_\Gamma$ is Cohen--Macaulay.
\end{corollary}

\begin{proof}[Proof of Theorem~\ref{thm:cm toric}]
Let $A=\k[G_L^{\ab}]$, with the natural action of $G_L$.  Then $A$ is a
free module over $\k[G_\tau]$, for any $\tau$, hence Cohen--Macaulay.
So if $L$ is a Cohen--Macaulay complex over $\k$, Theorem~7.4 of
\cite{DSY16} states that $H^p(T_L,A)=0$ for all $p \neq n+1$, and 
$H^{n+1}(T_L,A)$ is free, so $T_L$ is an abelian duality space.

To prove the converse, we recall the proof of that result uses a 
spectral sequence with $E_2$ term given by
\begin{equation}
\label{eq:e2pq}
E_2^{pq}=
\bigoplus_{\substack{\tau\in L: \, q=2\abs{\tau}}}
A_{G_\tau}\otimes_\k \widetilde{H}^{p+\abs{\tau}-1}(\lk_L(\tau),\k).
\end{equation}
This is a spectral sequence of $A$-modules, and we
note that $A_{G_\tau}$ is has codimension-$\abs{\tau}$, since
it is a quotient of $A$ by a regular sequence of length $\abs{\tau}$.
In particular, $A_{G_\tau}$ is nonzero, so from \eqref{eq:e2pq}, 
$L$ is Cohen--Macaulay over $\k$ if and only if $E^{pq}_2=0$ whenever
$p+q\neq n+1$.  

Now suppose that $H^p(T_L,A)=0$ for all $p\neq n+1$, so $E^{pq}_\infty=0$
for $p+q\neq n+1$.  If $L$ is not Cohen--Macaulay, then, choose the least
integer $q$ for which $E^{pq}_2\neq0$ for some $p$ where $p<n+1-q$.  By
our remarks above, $E^{pq}_2$ is an $A$-module of codimension $q/2$. 
Since $E^{pq}_\infty=0$, though, $E^{pq}_2$ must be filtered by kernels
of differentials $d^{pq}_r$ for $r\geq 2$.  The targets of such differentials
all have strictly higher codimension as $A$-modules, though, which gives
a contradiction.
\end{proof}

The linearized version behaves in the same way.
\begin{theorem}
\label{thm:cm-epy}
The cohomology algebra $A=H^\hdot(T_L,\k)$ has the EPY property
if and only if the complex $L$ is Cohen--Macaulay over $\k$.
\end{theorem}

\begin{proof}
Set $n=1+\dim L$.  Let $E$ be the exterior algebra over $H^1(T_L,\k)$, 
let $S$ be the symmetric algebra over $H^1(T_L,\k)$, and 
let $J_L\subset S$ be the symmetric Stanley--Reisner ideal.
By the construction of Aramova, Avramov and Herzog~\cite{AAH99}, 
$A^*(n)$ has a linear free resolution (over $E$) for some $n$ if and only the
$J_{L^*}(n)$ has a linear free resolution (over $S$), where $L^*$ 
denotes the Alexander dual.  By a result of Eagon and 
Reiner~\cite[Thm.\ 3]{ER98}, 
this is the case if and only if $L$ is Cohen--Macaulay.
\end{proof}

\begin{remark}
\label{rem:characteristic}
The Cohen--Macaulay property over a field $\k$ depends on the characteristic.
For example, if $L$ is a flag triangulation of ${\mathbb{RP}}^2$, then 
examining the link of the empty simplex shows $L$ is not Cohen--Macaulay
(integrally), though $L$ is Cohen--Macaulay over any field $\k$ of 
characteristic except $2$.  It follows from Theorem~\ref{thm:cm toric} that
$T_L$ is not an abelian duality space (integrally), though it is an 
abelian duality space over a field $\k$ with $\ch\k\neq 2$.  

Since $T_L$ has dimension $3$ and $H^1(T_L,\Z)$ is torsion-free, we conclude
that $H^2(T_L,\Z[G^{\ab}])$ must be a nonzero abelian $2$-group:
direct calculation shows, in fact, that
$H^2(T_L,\k[G^{\ab}])\cong\k[G^{\ab}]\neq0$ if $\ch\k=2$.
Analogously, the EPY property holds except in characteristic $2$, 
a phenomenon observed first (in the symmetric Stanley--Reisner setting)
in \cite{ER98}.
\end{remark}

\subsection{Propagation}
\label{subsect:toric_prop}
If the simplicial complex $L$ is Cohen--Macaulay, then the characteristic and 
resonance varieties of the toric complex $T_L$ (and the right-angled Artin 
group $G_L$) propagate, by Theorems~\ref{th:prop} and \ref{thm:resprop}.  
In view of Theorem \ref{thm:cjl tc}, this has the following, purely
combinatorial interpretation.

\begin{corollary}
\label{cor:cmtes}
Let $L$ be a Cohen--Macaulay complex over $\k$.  Suppose there is a 
subset $\sW\subset \sV$ of the vertex set and a simplex $\sigma$ 
supported on $\sV\setminus \sW$ such that 
$\widetilde{H}_{i-1-\abs{\sigma}}(\lk_{L_{\sW}}(\sigma),\k)\ne 0$, 
for some $i\ge \abs{\sigma}$. Then, for all $i\le j\le \dim(L) +1$, 
there exists a subset $\sW\subset \sW' \subset \sV$ and a simplex 
$\sigma'$ supported on $\sV\setminus \sW'$ such that 
$\widetilde{H}_{j-1-\abs{\sigma'}}(\lk_{L_{\sW'}}(\sigma'),\k)\ne 0$.
\end{corollary}

\begin{question}
Is there a direct, combinatorial proof of the result above?  
\end{question}

As noted in \cite[\S3.5]{PS09}, the resonance varieties of toric 
complexes do not always propagate.  For instance, if 
$\Gamma=\Gamma_1 \coprod \Gamma_2$, where 
$\Gamma_{j}=K_{m_{j}}$ are complete graphs on $m_j\ge 2$ vertices, 
$j=1,2$, and $G_{\Gamma}$ is the corresponding right-angled 
Artin group, then $\RR^1(G_{\Gamma},\k) = \k^{m_1+m_2}$, yet 
$\RR^i(G_{\Gamma},\k) = \k^{m_1}\times \set{0} \cup \set{0}\times \k^{m_2}$ 
for $1<i\le \min (m_1, m_2)$.

\begin{question}
\label{quest:non-cm}
Is there a finite simplicial complex $L$ which is not Cohen--Macaulay 
but for which the resonance varieties of $T_L$ still propagate?
\end{question}

\begin{remark}
\label{rem:huh}
We saw in Proposition~\ref{prop:EPYbounds}
that the Betti numbers of spaces with the 
EPY property satisfy certain inequalities, since they are the Betti numbers 
of a minimal linear resolution.  By a result of Huh~\cite{Huh15}, the Betti 
numbers of (linear) hyperplane arrangement complements are log-concave, that is, 
$b_{i-1}b_{i+1}\leq b_i^2$ for $1\leq i\leq n-1$, which implies, in particular,
that they form a unimodal sequence.  

However, the example of toric complexes cautions us not to expect much 
in general.  The Betti numbers of $T_L$ are given by the $f$-vector of 
$L$, and an example of Bj\"orner~\cite{Bj81} shows that, even in the
case of a simplicial sphere, the $f$-vector need not be unimodal.
\end{remark}

\subsection{More general polyhedral products}
\label{subsec:gggmacs}
Here, we briefly consider groups that arise from more general parameters
in the polyhedral product construction.  That is, we 
let $X_i=\bigvee_{r_i} S^1$ for $1\leq i\leq m$, and consider the 
space $\ZZ_L(\X,*)$.  This is an instance of the following operadic 
construction for polyhedral products, first considered by 
Ayzenberg in \cite{Az13}.

\begin{definition}
\label{def:simpcomp}
Suppose $L$ is a simplicial complex on the vertex set $[m]$, for some
$n\geq1$, and $K_1,\ldots,K_m$ are simplicial complexes on disjoint 
vertex sets $V_1,\ldots,V_m$, respectively.  We define the composition
$L\circ (K_1,\ldots,K_m)$ to be the simplicial complex with vertices
$V:=\bigcup_{i=1}^m V_i$, whose simplices are all sets of the form
\begin{equation}
\label{eq:tauspace}
\bigcup_{\substack{i\colon i\in \tau,\\ \sigma_i\in K_i}} \sigma_i,
\end{equation}
where $\tau\in L$.  Equivalently, $L\circ(K_1,\ldots,K_m)$ is the
union of the simplicial
joins $K_{i_1}\star\cdots\star K_{i_k}$ over all simplices $\tau=\set{i_1,
\ldots,i_k}$ of $L$.
\end{definition}

\begin{example}
\label{ex:wedge}
If $(j_1,\ldots,j_m)$ are positive integers and $\Delta^k$ denotes the
$k$-dimensional simplex, then
$L^*\circ(\Delta^{j_1-1},\ldots,\Delta^{j_m-1})=L(J)^*$, where $-^*$ here
denotes Alexander duality, and $L(J)$ is the simplicial wedge construction
introduced in \cite{BBCG15}.
\end{example}

We will need the following, which is similar to a result of
Ayzenberg \cite[Prop.~5.1]{Az13}.  For the convenience 
of the reader, we supply a short, self-contained proof.  

\begin{proposition} 
\label{prop:ZKoperad}
Suppose $L$ and $K_1,\ldots,K_n$ are simplicial complexes as above.
Suppose $(X_{ij},A_{ij})$ is a pair of (nonempty) finite CW-complexes,
for each $1\leq i\leq n$ and $j\in V_i$.
Let $\X:=(X_{ij})_{1\leq i\leq m, j\in V_i}$
and $\X_i=(X_{ij})_{j\in V_i}$
for each $1\leq i\leq m$.  Define $\AA$ and $\AA_i$ similarly.
For each $i$, let $Y_i=\ZZ_{K_i}(\X_i,\AA_i)$, and
$B_i=\prod_{j\in V_i}A_{ij}$.
Then there is a natural homeomorphism of CW-complexes, 
\[
\ZZ_L(\underline{Y},\underline{B})
\cong
\ZZ_{L\circ (K_1,\ldots,K_m)}(\X,\AA).
\]
\end{proposition}

\begin{proof}
Both spaces are subcomplexes of $\prod_{i=1}^m\prod_{j\in V_j}X_{ij}$, 
constructed as unions of products.  To see they are unions of the same products, 
consider a simplex $S$ of $L\circ(K_1,\ldots,K_m)$.  By construction,
$S=\bigcup_{i\in\tau}\sigma_i$ for some $\tau\in L$ and choice of 
$\sigma_i \in K_i$, for each $i\in \tau$.

By definition, $(\X,\AA)^S=\prod_{j\in V}\X_{S,j}$, where
\begin{equation}
\label{eq:ZKoperad}
\X_{S,j}=\begin{cases}
X_{ij} & \text{if $j\in \sigma_i$ and $i\in\tau$;}\\
A_{ij} & \text{if not.}
\end{cases}
\end{equation}
(In the first case, we recall that $i$ is (uniquely) determined by $j$,
since $V=\coprod_i V_i$.)

On the other hand, $(\underline{Y},\underline{B})^\tau=\prod_{i=1}^m
\underline{Y}_{\tau,i}$, where 
\begin{equation}
\label{eq:ZKguy}
\underline{Y}_{\tau,i}=\begin{cases}
\ZZ_{K_i}(\X_i,\AA_i) & \text{if $i\in \tau$, and}\\
B_i & \text{if not.}
\end{cases}
\end{equation}
Note that $\underline{Y}_{\tau,i}\subseteq \prod_{j\in V_i}X_{ij}$, for
each $i$.  Comparing its $j$th coordinate with \eqref{eq:ZKoperad},
we see the two are equal.
\end{proof}

\subsection{Reinterpreting Cohen--Macaulayness}
\label{subsec:cm gms}
Now we return to the goal of this section.  Given a simplicial complex
$L$ with $m$ vertices, let $K_i$ be the zero-dimensional
complex with $r_i$ vertices.  Let $X_i=\ZZ_{K_i}(S^1,*)\cong\bigvee_{r_i}S^1$.
Then Proposition~\ref{prop:ZKoperad} shows that $\ZZ_L(\X,*)$ is homeomorphic
to a toric complex on $N=\sum_{i=1}^m r_i$ vertices,
\begin{equation}
\label{eq:ZLguy}
\ZZ_L(\X,*)\cong \ZZ_{L\circ(K_1,\ldots,K_m)}(S^1,*).
\end{equation}

The simplicial complex $L(r_1,\ldots,r_m):=L\circ(K_1,\ldots,K_m)$ 
is obtained from $L$
by letting $V_i$ be (disjoint) sets of $r_i$ vertices for each $i$, 
and declaring $S$ to be a simplex of $\widetilde{L}$ if and only if
$\set{i\colon V_i\cap S\neq \emptyset}$ is a simplex of $L$ of the same
dimension as $S$.

We obtain the following generalization of Theorem~\ref{thm:cm toric}.

\begin{theorem}
The polyhedral product space $\ZZ_L(\X,*)$, where $X_i=\bigvee_{r_i}S^1$
for $1\leq i\leq m$ and some positive integers $r_1,\ldots,r_m$,
is an abelian duality space if and only if $L$ is Cohen--Macaulay.
\end{theorem}

\begin{proof}
Following the approach for toric complexes, we cover $\ZZ_L(\X,*)$ with
the sets $(\X,*)^\tau$, for each $\tau\in L$.  The inclusion of 
$(\X,*)^\tau\in\ZZ_L(\X,*)$ is easily seen to be a retract, so 
$G_\tau := \pi_1((\X,*)^\tau,*)$ is a subgroup of $G := \pi_1(\ZZ_L(\X,*),*)$.
Each $G_\tau$ is a product of free groups, hence an abelian duality group
of dimension $\abs{\tau}$.  If $A$ is a $\Z[G]$-module we obtain a 
spectral sequence as in \cite{DSY16} with
\begin{equation}
\label{eq:e2pq-bis}
E_2^{pq}=\widetilde{H}^{p+\abs{\tau}-1}(\lk_L(\tau),H^{q-\abs{\tau}}(G_\tau,A)).
\end{equation}
Since $A=\Z[G^{\ab}]$ 
is a free $\Z[G]$-module, $E_2^{pq}=0$ unless $q=2\abs{\tau}$.  If $L$ is
Cohen--Macaulay of dimension $n$, $E_2^{pq}=0$ unless 
$p+\abs{\tau}-1=n-\abs{\tau}$.  Combining,
we see the spectral sequence degenerates at $E_2$ and gives the desired
conclusion.

The converse is proven by the same argument as Theorem~\ref{thm:cm toric}.
\end{proof}

\subsection{Bestvina--Brady groups}
\label{subsec:bb}
Similar questions can be asked about the Bestvina--Brady groups 
\begin{equation}
\label{eq:bb}
N_{\Gamma}=\ker (\nu \colon G_{\Gamma}\to \Z),
\end{equation}
where $\nu$ is the diagonal character, taking each standard 
generators $v\in G_{\Gamma}$ to $1\in \Z$.  We may view these groups as
fundamental groups of $\Z$-covers of toric complexes which are
directly analogous to the Milnor fibration, viewed as a $\Z$-cover,
in the previous section.

Once again, we find that the covers share the duality properties
of the base, subject to some necessary additional hypothesis.
As shown by Bestvina and Brady in \cite{BB}, the group $N_\Gamma$ 
is of type FP if and only if the flag complex $\Delta_\Gamma$ is acyclic,
in which case its cohomological dimension is the dimension of 
$\Delta_\Gamma$.  With this in mind, (classical) duality is characterized
by the following recent result of Davis and Okun. 

\begin{proposition}[\cite{DO12}, Prop.~9.3]   
\label{prop:cm bb}
Suppose $\Delta_{\Gamma}$ is $\k$-acyclic.  Then $N_{\Gamma}$ is a 
duality group if and only if $\Delta_{\Gamma}$ is Cohen--Macaulay over $\k$. 
\end{proposition}

\begin{theorem}
\label{thm:bb cm}
Let $\Gamma$ be a graph, and $\Delta_\Gamma$ its flag complex.  
Then $N_\Gamma$ is an abelian duality group 
if and only if $\Delta_\Gamma$ is acyclic and Cohen--Macaulay.
\end{theorem}

\begin{proof}
Suppose $\Delta_\Gamma$ is an acyclic flag complex.  By \cite{BB}, 
$N_\Gamma$ is of type FP of dimension $n=\dim \Delta_\Gamma$.
In particular, $\Delta_\Gamma$ is connected, 
so, by \cite[Lem.~3.2]{PS07}, the sequence
\begin{equation}
\label{eq:bbseq}
\xymatrix{1\ar[r]& N_\Gamma \ar[r]& G_\Gamma \ar[r]& \Z \ar[r]& 0}
\end{equation}
is $\ab$-exact, in the sense of Definition~\ref{def:goodH1}. 
Now further suppose that $\Delta_\Gamma$ is also Cohen--Macaulay. Then, 
by Proposition \ref{prop:cm raag}, $G_\Gamma$ is an abelian duality group. 
Hence, by Proposition~\ref{prop:extensions}\eqref{prop:extn:3}, 
$N_\Gamma$ is also an abelian duality group.

Conversely, suppose $N_\Gamma$ is an abelian duality group.  
Then $N_\Gamma$ is of type FP, and so $\Delta_\Gamma$ must be acyclic.  
By Proposition~\ref{prop:extensions}\eqref{prop:extn:1},
we see that $G_\Gamma$ is also an abelian duality group, 
whence $\Delta_\Gamma$ is Cohen--Macaulay, by 
Theorem~\ref{thm:cm toric}.
\end{proof}
Once again, we find that the EPY property follows in step with
abelian duality.

\begin{theorem}
\label{thm:bbepy}
If $\Delta_\Gamma$ is $\k$-acyclic and Cohen--Macaulay over $\k$, 
then $H^\hdot(N_\Gamma,\k)$ has the EPY property.
\end{theorem}

\begin{proof}  
Let $L=\Delta_\Gamma$, so that $G_\Gamma=\pi_1(T_L)$, and let 
$N_\Gamma=\pi_1(\widetilde{T_L})$, the $\Z$-cover of the toric complex 
classified by $a=\sum_{i=1}^n e_i\in H^1(T_L,\Z)$.
From Theorem~\ref{thm:cjl tc}, we note that $a\not\in\RR^1(T_L)$.

Recall that $A=H^\hdot(T_L,\k)\cong E/J_L$, an exterior Stanley--Reisner ring.  
Since, by assumption, $\widetilde{H}_{\hdot} (L,\k)=0$, we may apply 
\cite[Cor.~7.2]{PS09} to conclude that the cohomology ring $B := H^\hdot(\widetilde{T_L},\k)$ 
is isomorphic to $E/(J_L+(a))$.  Multiplication by $a$ gives a surjective
homomorphism of graded $E$-modules, 
$E/(J_L+(a))(1)\to (J_L+(a))/J_L$; because $a$ is nonresonant,
this homomorphism is an isomorphism.  We thus obtain an exact 
sequence of $E$-modules, 
\begin{equation}
\label{eq:bab}
\xymatrix@C-0.5em{
0\ar[r] & B(1)\ar[r] & A\ar[r] & B\ar[r] & 0},
\end{equation}
hence also an exact sequence 
\begin{equation}
\label{eq:babstar}
\xymatrix@C-0.5em{
0\ar[r] & B^*(n+1)\ar[r] & A^*(n+1)\ar[r] & B^*(n)\ar[r] & 0},
\end{equation}
where $n=\dim L$, 
by taking duals and shifting.  

By hypothesis, $A^*(n+1)$ is a Koszul module,
which is to say $\Tor^E_{p}(A^*(n+1),\k)^q$ is concentrated in degree $q=p$.
By examining the graded long exact sequence of $\Tor^E_\hdot(-,\k)$, 
we see by induction that $B^*(n)$ is also a Koszul $E$-module.  A
change-of-rings argument shows that $B^*(n)$ is Koszul as an $E/(a)$-module
as well, which completes the proof.
\end{proof}

\newcommand{\arxiv}[1]
{\texttt{\href{http://arxiv.org/abs/#1}{arXiv:#1}}}
\newcommand{\arx}[1]
{\texttt{\href{http://arxiv.org/abs/#1}{arxiv:}}
\texttt{\href{http://arxiv.org/abs/#1}{#1}}}
\newcommand{\doi}[1]
{\texttt{\href{http://dx.doi.org/#1}{doi:#1}}}
\renewcommand{\MR}[1]
{\href{http://www.ams.org/mathscinet-getitem?mr=#1}{MR#1}}
\newcommand{\MRh}[2]
{\href{http://www.ams.org/mathscinet-getitem?mr=#1}{MR#1 (#2)}}

\end{document}